\documentclass[12pt]{amsart}
\usepackage[all]{xy}

\usepackage{amsmath, amsthm, amssymb,latexsym, graphics}
\usepackage{calc,color,sidecap,caption2}

\newcommand{\N}{\mathbb{N}}
\newcommand{\Z}{\mathbb{Z}}

\newcommand{\R}{\mathbb{R}}
\newcommand{\C}{\mathbb{C}}
\newcommand{\E}{\mathbb{E}}
\newcommand{\X}{\mathcal{X}}

\newcommand{\spr}[2]{\langle #1, #2 \rangle}
\newcommand{\ck}{\check{\phantom{i}}}

\newcommand{\HI}{H^\infty}
\newcommand{\Bes}{\mathcal{B}}
\newcommand{\Mih}{\mathcal{M}}
\newcommand{\Bbii}{\Bes^\beta_{\infty,\infty}}
\newcommand{\Ba}{\Bes^\alpha_{\infty,1}}
\newcommand{\Ma}{\Mih^\alpha}
\DeclareMathOperator{\loc}{loc}
\newcommand{\Baloc}{\Bes^\alpha_{\infty,1,\loc}}

\newcommand{\Fdyad}{\phi}
\newcommand{\dyad}{\varphi}
\newcommand{\equi}{\psi}

\newcommand{\tequi}{\widetilde{\equi}}
\newcommand{\ddyad}{\dot{\dyad}}

\newcommand{\dB}{\dot{B}}
\newcommand{\B}{B}

\DeclareMathOperator{\supp}{supp}
\DeclareMathOperator{\Id}{id}

\DeclareMathOperator{\Str}{Str}
\DeclareMathOperator{\Hol}{Hol}

\let\Re=\relax \DeclareMathOperator{\Re}{Re}
\let\Im=\relax \DeclareMathOperator{\Im}{Im}

\newcommand{\bignorm}[1]{\Bigl\Vert#1\Bigr\Vert}
\newcommand{\Bignorm}[1]{\Bigl\Vert#1\Bigr\Vert}

\newtheorem{thmalt}{Theorem}[section]

\theoremstyle{definition}
\newtheorem{rem}[thmalt]{Remark}
\newtheorem{defi}[thmalt]{Definition}
\newtheorem{thm}[thmalt]{Theorem}
\newtheorem{cor}[thmalt]{Corollary}
\newtheorem{lem}[thmalt]{Lemma}
\newtheorem{prop}[thmalt]{Proposition}
\newtheorem{nota}[thmalt]{Notation}

\numberwithin{equation}{section}

\title[Paley-Littlewood Decomposition for Sectorial Operators]
 {Paley-Littlewood Decomposition for Sectorial Operators and Interpolation Spaces} 

\author[Ch. Kriegler]{Christoph Kriegler}
\address{Christoph Kriegler\\
Laboratoire de Math\'ematiques (CNRS UMR 6620)\\
Universit\'e Blaise-Pascal (Clermont-Ferrand 2)\\
Campus des C\'ezeaux\\
3, place Vasarely\\
TSA 60026\\
CS 60026\\
63 178 Aubi\`ere Cedex\\
France}

\email{christoph.kriegler@math.univ-bpclermont.fr}
\thanks{The first named author acknowledges financial support from the Franco-German University (DFH-UFA) and the Karlsruhe House of Young Scientists (KHYS).
The second named author acknowledges the support by the DFG through CRC 1173.}

\author[L. Weis]{Lutz Weis}
\address{Lutz Weis\\
Karlsruher Institut f\"ur Technologie\\
Fakult\"at f\"ur Mathematik\\
Institut f\"ur Analysis\\
Englerstra\ss{}e 2\\
76131 Karlsruhe\\
Germany}
\email{lutz.weis@kit.edu}

\date{\today}

\allowdisplaybreaks

\begin{document}

\begin{abstract}
We prove Paley-Littlewood decompositions for the scales of fractional powers of $0$-sectorial operators $A$ on a Banach space which correspond to Triebel-Lizorkin spaces and the scale of Besov spaces if $A$ is the classical Laplace operator on $L^p(\R^n).$
We use the $\HI$-calculus, spectral multiplier theorems and generalized square functions on Banach spaces and apply our results to Laplace-type operators on manifolds and graphs, Schr\"odinger operators and Hermite expansion.
We also give variants of these results for bisectorial operators and for generators of groups with a bounded $\HI$-calculus on strips.\end{abstract}

\maketitle

\section{Introduction}\label{Sec 1 Intro}

Littlewood-Paley decompositions do not only play an important role in the theory of the classical Besov and Triebel-Lizorkin spaces but also in the study of scales of function spaces associated with Laplace-Beltrami operators on manifolds, Laplace-type operators on graphs and fractals, Schr\"odinger operators and operators associated with various orthogonal expansions (Hermite, Laguerre, etc).
They are an important tool in the more detailed analysis of these function spaces (e.g. wavelet and ``molecular'' decompositions) but also in the study of partial differential equations (see e.g. \cite{BCD,IP,T,BGT}).

The most common approach in the literature is to start from a self-adjoint operator on an $L^2(U)$-space and then use extrapolation techniques (e.g. transference principles, Gaussian bounds, off-diagonal estimates) and interpolation to obtain Paley-Littlewood decompositions on the $L^p(U)$-scale.
In this paper we offer a more general approach which does not assume that $A$ is defined on an $L^p$-scale nor that $A$ is self-adjoint in some sense.
We consider $0$-sectorial operators $A$ on a general Banach space and construct Paley-Littlewood decompositions for the scale of fractional domains of $A$ (also of negative order) using as our main tool the $\HI$ functional calculus and spectral multiplier theorems connected with it.
Since it is well-known that the classical operators mentioned above do have a bounded $\HI$-calculus we offer a unified approach which covers the known results for various classes of operators.

Let us describe a typical result for a $0$-sectorial operator $A$ on a Banach space $X.$
We assume that $A$ has a bounded $\HI(\Sigma_\omega)$ calculus on sectors $\Sigma_\omega$ of angle $\omega$ around $\R_+$ for all $\omega > 0,$ i.e. there are constants $C < \infty,\: \alpha > 0$ such that $\|f(A)\| \leq \frac{C}{\omega^\alpha} \sup_{\lambda \in \Sigma_\omega} |f(\lambda)|$ for all bounded analytic functions $f \in \HI(\Sigma_\omega)$ and $\omega > 0.$
This corresponds to the boundedness of spectral multipliers $f(A)$ on $X,$ where $f$ satisfies the Mihlin condition
\begin{equation}\label{Equ Intro Mihlin condition} \sup_{t > 0} |t^k f^{(k)}(t)| < \infty \quad (k = 0,1,\ldots,N).
\end{equation}
Let $\ddyad_n$ be a sequence in $C^\infty_c(\R_+)$ satisfying $\supp \ddyad_0 \subset [\frac12,2],\: \ddyad_n = \ddyad_0(2^{-n}(\cdot)),\: \sum_{n \in \Z} \ddyad_n(t) = 1$ and $(\epsilon_n)_{n \in \Z}$ an independent sequence of Bernoulli random variables (or Rademacher functions).
Then for all $\theta \in \R$ and $x \in D(A^\theta)$ 
\begin{equation}\label{Equ Intro fractional} \| A^\theta x \| \cong \E_\omega \| \sum_{n \in \Z} \epsilon_n(\omega) 2^{n \theta} \ddyad_n(A) x \|
\end{equation}
where $\ddyad_n(A)$ is defined as a ``spectral multiplier''.
(See Section \ref{Sec 2 Prelims} for the necessary background in spectral theory.)
If $X$ is an $L^p(U)$ space then by Kahane's inequality, the random sum in \eqref{Equ Intro fractional} reduces to the classical square sum 
\[\|A^\theta x \| \cong \| (\sum_{n \in \Z} |2^{n \theta} \ddyad_n(A) x|^2 )^{\frac12} \|_{L^p(U)}\] 
Therefore the random sums in \eqref{Equ Intro fractional} can be regarded as the natural extension of $L^p$-square sums in \eqref{Equ Intro fractional} to the general Banach space setting (see Subsection \ref{Subsec Prelims Rad}).
The random sums are a useful and necessary tool in the context of Bochner spaces, Sobolev spaces of Banach space valued functions and mixed norm spaces.
For $p > 2,$ this implies by Minkowski's inequality the estimate
\[ \| A^\theta x \| \leq C \left( \sum_{n \in \Z} \| 2^{n\theta} \ddyad_n(A) x \|^2 \right)^{\frac12} \]
which has proven to be very useful in the theory of dissipative evolution equations (see e.g. \cite{IP} and the literature quoted there).
In Section \ref{Sec 3 Spectral Decomposition} we prove as one of our main results statement \eqref{Equ Intro fractional} and a companion result concerning the inhomogeneous Paley-Littlewood decomposition.
Furthermore we give ``continuous'' versions based on generalized square functions.
For $X = L^p(U),$ they read as
\[ \| A^\theta x \| \cong \left\| \left( \int_0^\infty | t^{-\theta} \psi(tA) x |^2 \frac{dt}{t} \right)^{\frac12} \right\|_{L^p(U)}. \]
We also consider decomposing functions $\ddyad_n$ with less regularity than $C^\infty.$

These results are proven using a ``Mihlin'' functional calculus for functions satisfying \eqref{Equ Intro Mihlin condition}.
However, in a sense the Paley-Littlewood decomposition \eqref{Equ Intro fractional} is equivalent to the boundedness of a Mihlin functional calculus:
Assuming a weaker functional calculus and the Paley-Littlewood decomposition \eqref{Equ Intro fractional}, we show the boundedness of a Mihlin functional calculus (see Proposition \ref{Prop weak to strong calculus}).

Whereas the results in Section \ref{Sec 3 Spectral Decomposition} are modelled after the Triebel-Lizorkin spaces, we consider in Section \ref{Sec 5 Real Interpolation} analogues of Besov spaces which are based on the real interpolation method, e.g. $\dB^\theta_q = (X,D(A))_{\theta,q}.$
We show in Theorems \ref{Thm Besov identification} and \ref{Thm Besov identification continuous}
\[ \|x\|_{\dB^\theta_q} \cong \left( \sum_{n \in \Z} 2^{n\theta q} \|\ddyad_n(A)x\|_X^q \right)^{\frac1q}, \quad \|x\|_{\dB^\theta_q} \cong \left( \int_0^\infty t^{-\theta q} \| \ddyad_0(tA)x\|^q \frac{dt}{t} \right)^{\frac1q} .\]
These results can be obtained under much weaker assumptions and do not require a bounded $\HI$ calculus for $A.$
In particular we show that a weak functional calculus for $A$ on $X$ suffices to obtain automatically a Mihlin functional calculus on these Besov type spaces.

In Section \ref{Sec 6 Examples} we apply our results to various classes of operators obtaining new Paley-Littlewood decompositions but also recovering many results known in the literature in a unified way.

In Section \ref{Sec 6 Bisectorial Operators} we indicate how to extend our results to bisectorial operators.

If the sectorial operator $A$ has bounded imaginary powers then the results of Sections \ref{Sec 3 Spectral Decomposition} and \ref{Sec 5 Real Interpolation} translate into decompositions for the group $U(t) = A^{it}$ with generator $B = \log(A).$
In Section \ref{Sec 6 Strip-type Operators} we sketch the corresponding decompositions for $B$,
under the assumption that $B$ is a strip-type operator with a bounded $\HI$ calculus on each strip $\Str_\omega$ around $\R$ for $\omega > 0.$

\section{Preliminaries}\label{Sec 2 Prelims}

\subsection{$0$-sectorial operators}\label{Subsec A B}

We briefly recall standard notions on $\HI$ calculus.
For $\omega \in (0,\pi)$ we let $\Sigma_\omega = \{ z \in \C \backslash \{ 0 \} :\: | \arg z | < \omega \}$ the sector around the positive axis of aperture angle $2 \omega.$
We further define $\HI(\Sigma_\omega)$ to be the space of bounded holomorphic functions on $\Sigma_\omega.$
This space is a Banach algebra when equipped with the norm $\|f\|_{\infty,\omega} = \sup_{\lambda \in \Sigma_\omega} |f(\lambda)|.$

A closed operator $A : D(A) \subset X \to X$ is called $\omega$-sectorial, if the spectrum $\sigma(A)$ is contained in $\overline{\Sigma_\omega},$ $R(A)$ is dense in $X$ and
\begin{equation}\label{Equ Def Sectorial}
\text{for all }\theta > \omega\text{ there is a }C_\theta > 0\text{ such that }\|\lambda (\lambda - A)^{-1}\| \leq C_\theta \text{ for all }\lambda \in \overline{\Sigma_\theta}^c.
\end{equation}
Here, $\overline{\Sigma_\theta}^c = \C \backslash \overline{\Sigma_\theta}$ is the set complement.
Note that $\overline{R(A)} = X$ along with \eqref{Equ Def Sectorial} implies that $A$ is injective.
In the literature, the condition $\overline{R(A)} = X$ is sometimes omitted in the definition of sectoriality.
Note that if $A$ satisfies the conditions defining $\omega$-sectoriality except $\overline{R(A)} = X$ on $X = L^p(\Omega),\, 1 < p < \infty$ (or any reflexive space),
then there is a canonical decomposition 
\begin{equation}\label{Equ sectorial injective}
X  = \overline{R(A)} \oplus N(A),\:x = x_1 \oplus x_2,\text{ and }A = A_1 \oplus 0,\,x \mapsto A x_1 \oplus 0,
\end{equation}
such that $A_1$ is $\omega$-sectorial on the space $\overline{R(A)}$ with domain $D(A_1) = \overline{R(A)} \cap D(A).$

For $\theta \in (0,\pi),$ we let
$\HI_0(\Sigma_\theta) = \{ f \in \HI(\Sigma_\theta):\: \exists\: C,\epsilon > 0:\: |f(\lambda)| \leq C (1+|\log \lambda|)^{-1-\epsilon} \}.$
Note that in the literature, this space is usually defined slightly differently, imposing the decay $|f(\lambda)| \leq C |\lambda|^\epsilon / | 1 + \lambda |^{2\epsilon}.$
Our space is larger and the reason for the different choice is of minor technical nature.
Then for an $\omega$-sectorial operator $A$ and a function $f \in \HI_0(\Sigma_\theta)$ for some $\theta \in (\omega,\pi),$ one defines the operator
\begin{equation}\label{Equ Cauchy Integral Formula}
f(A) = \frac{1}{2 \pi i} \int_{\Gamma} f(\lambda) (\lambda - A)^{-1} d\lambda ,
\end{equation}
where $\Gamma$ is the boundary of a sector $\Sigma_\sigma$ with $\sigma \in (\omega,\theta),$ oriented counterclockwise.
By the estimate of $f,$ the integral converges in norm and defines a bounded operator.
If moreover there is an estimate $\|f(A)\| \leq C \|f\|_{\infty,\theta}$ with $C$ uniform over all such functions, then $A$ is said to have a bounded $\HI(\Sigma_\theta)$ calculus.
In this case, there exists a bounded homomorphism $\HI(\Sigma_\theta) \to B(X),\,f \mapsto f(A)$ extending the Cauchy integral formula \eqref{Equ Cauchy Integral Formula}.

We refer to \cite{CDMY} for details.
We call $A$ $0$-sectorial if $A$ is $\omega$-sectorial for all $\omega > 0.$
For $\omega \in (0,\pi),$ define the algebra of functions $\Hol(\Sigma_\omega) = \{ f : \Sigma_\omega \to \C :\: \exists \: n \in \N :\: \rho^n f \in \HI(\Sigma_\omega) \},$
where $\rho(\lambda) = \lambda (1 + \lambda)^{-2}.$ 
For a proof of the following lemma, we refer to \cite[Section 15B]{KuWe} and \cite[p.~91-96]{Haasa}.

\begin{lem}\label{Lem Hol}
Let $A$ be a $0$-sectorial operator.
There exists a linear mapping, called the extended holomorphic calculus,
\begin{equation}\label{Equ Extended HI calculus}
\bigcup_{\omega > 0} \Hol(\Sigma_\omega) \to \{ \text{closed and densely defined operators on }X \},\: f \mapsto f(A)
\end{equation}
extending \eqref{Equ Cauchy Integral Formula} such that for any $f,g \in \Hol(\Sigma_\omega),$ $f(A)g(A)x = (f g)(A)x$ for $x \in \{y \in D(g(A)):\: g(A)y \in D(f(A)) \} \subset D((fg)(A))$ and
$D(f(A)) = \{ x \in X :\: (\rho^n f)(A) x \in D(\rho(A)^{-n}) = D(A^n) \cap R(A^n) \},$ where $(\rho^n f)(A)$ is given by \eqref{Equ Cauchy Integral Formula}, i.e. $n \in \N$ is sufficiently large.
\end{lem}

\subsection{Function spaces on the line and half-line}\label{Subsec Prelims Function spaces}

In this subsection, we introduce several spaces of differentiable functions on $\R_+ = (0,\infty)$ and $\R.$
Different partitions of unity play a key role.

\begin{defi}\label{Def Partitions of unity}~
\begin{enumerate}
\item
If $\ddyad \in C^\infty_c(0,\infty)$ with $\supp \ddyad \subset [\frac12, 2]$ and $\sum_{n=-\infty}^\infty \ddyad(2^{-n} t) = 1$ for all $t > 0,$
we put $\ddyad_n = \ddyad(2^{-n} \cdot)$ and call $(\ddyad_n)_{n \in \Z}$ a (homogeneous) dyadic partition of unity on $\R_+$.
\item
If $(\ddyad_n)_{n \in \Z}$ is a homogeneous dyadic partition of unity on $\R_+$, we put $\dyad_n = \ddyad_n$ for $n \geq 1$ and $\dyad_0 = \sum_{k = -\infty}^0 \ddyad_k,$ so that $\supp \dyad_0 \subset (0,2].$
Then we call $(\dyad_n)_{n \in \N_0}$ an inhomogeneous dyadic partition of unity on $\R_+.$
\item Let $\Fdyad_0,\,\Fdyad_1 \in C^\infty_c(\R)$ such that $\supp \Fdyad_1 \subset [\frac12,2]$ and $\supp \Fdyad_0 \subset [-1,1].$
For $n \geq 2,$ put $\Fdyad_n = \Fdyad_1(2^{1-n}\cdot),$ so that $\supp \Fdyad_n \subset [2^{n-2},2^n].$
For $n \leq -1,$ put $\Fdyad_n = \Fdyad_{-n}(-\cdot).$
We assume that $\sum_{n \in \Z} \Fdyad_n(t) = 1$ for all $t \in \R.$
Then we call $(\Fdyad_n)_{n \in \Z}$ a dyadic partition of unity on $\R$, which we will exclusively use to the decompose the Fourier image of a function.
\end{enumerate}

For the existence of such smooth partitions, we refer to the idea in \cite[Lemma 6.1.7]{BeL}.
In the later use of the above definitions (1) and (2) we could relax the condition of the functions belonging to $C^\infty_c(\R_+),$ to functions that are smooth up to an order $> \alpha,$ where $\alpha$ always denotes the derivation order of the functional calculus of $A.$
Whenever $(\phi_n)_n$ is a partition of unity as in (1) or (3), we put
\begin{equation}\label{Equ tilde partition}
\widetilde\phi_n = \sum_{k=-1}^1 \phi_{n+k}.
\end{equation}
It will be very often useful to note that
\begin{equation}\label{Equ Partition of unity}
\widetilde\phi_m \phi_n = \phi_n\text{ for }m=n\text{ and }\widetilde\phi_m \phi_n = 0\text{ for }|n-m| \geq 2.
\end{equation}
\end{defi}

We recall the following classical function spaces:
\begin{nota}\label{Def Classical function spaces}
Let $m \in \N_0$ and $\alpha > 0.$
\begin{enumerate}
\item $C^m_b = \{ f : \R \to \C:\: f \,m\text{-times differentiable and }f,f',\ldots,f^{(m)}$ 

$\text{ uniformly continuous and bounded}\}.$
\item $\Bes^\alpha_{p,q},$ the Besov spaces defined for example in \cite[p. 45]{Triea}:
Let $(\Fdyad_n)_{n\in\Z}$ be a dyadic partition of unity on $\R.$
Then
\[ \Bes^\alpha_{p,q} = \{ f \in C_b^0 : \: \|f\|_{B^\alpha_{p,q}}^q = \sum_{n \in \Z} 2^{|n|\alpha q} \|f \ast \check{\Fdyad_n}\|_p^q < \infty\}. \]
\end{enumerate}
The spaces in (2) are Banach algebras for any $\alpha > \frac1p$ \cite[p.~222]{RuSi}.

Further we also consider the local space
\begin{enumerate}
\item[(3)] $\Baloc = \{ f : \R \to \C :\: f\varphi \in \Ba\text{ for all }\varphi \in C^\infty_c\}$ for $\alpha > 0.$
\end{enumerate}
This space is closed under pointwise multiplication.
Indeed, if $\varphi \in C^\infty_c$ is given, choose $\psi \in C^\infty_c$ such that $\psi \varphi = \varphi.$
For $f,g \in \Baloc,$ we have $(fg)\varphi = (f \varphi)(g \psi) \in \Baloc.$
\end{nota}

\subsection{Fractional powers of sectorial operators}
We give a short overview on fractional powers.
Let $A$ be a $0$-sectorial operator and $\theta \in \C.$
Since $\lambda \mapsto \lambda^\theta \in \Hol(\Sigma_\omega)$ for $0 < \omega < \pi,$ $A^\theta$ is defined by the extended holomorphic calculus from \eqref{Equ Extended HI calculus}.
(If $A$ has an $\Ma_1$ calculus defined below, then $A^\theta$ is also given by the $\Ma_{\loc}$ calculus from Proposition \ref{Prop Soaloc calculus}.)
To $A^\theta,$ one associates the following two scales of extrapolation spaces (\cite[Definition 15.21, Lemma 15.22]{KuWe}, see also \cite[Section 2]{KaKW} and \cite[II.5]{EnNa})
\begin{align*}
\dot{X}_\theta & = (D(A^\theta),\|A^\theta \cdot\|_X)\,\widetilde{ }\,& (\theta \in \C),\\
\text{and }X_\theta & = (D(A^\theta),\|A^\theta \cdot\|_X + \|\cdot\|_X)\,\widetilde{ }\, & (\Re \theta \geq 0).
\end{align*}
Here, $\widetilde{  }\,$ denotes completion with respect to the indicated norm.
Clearly, $X = \dot{X}_0 = X_0.$
If $A = -\Delta$ is the Laplace operator on $L^p(\R^d),$ then $\dot{X}_\theta$ is the Riesz or homogeneous potential space, whereas $X_\theta$ is the Bessel or inhomogeneous potential space.
For two different values of $\theta,$ the completions can be realized in a common space.
More precisely, if $m \in \N,\,m \geq \max(|\theta_0|,|\theta_1|),$ then $\dot{X}_{\theta_j}$ and $X_{\theta_j}$ can be viewed as subspaces of
\begin{equation}\label{Equ The space Ym}
(X,\|(A(1+A)^{-2})^m \cdot\|_X)\,\widetilde{ }\,
\end{equation}
for $j=0,1.$
Then for $\theta > 0,$ one has $X_\theta = \dot{X}_\theta \cap X$ with equivalent norms \cite[Propositions 15.25 and 15.26]{KuWe}.
Thus, $\{\dot{X}_{\theta_0},\dot{X}_{\theta_1}\}$ and $\{X_{\theta_0},X_{\theta_1}\}$ form an interpolation couple.
It is known that if $A$ has bounded imaginary powers, then we have for the complex interpolation method
\begin{equation}\label{Equ Complex Interpolation BIP}
[\dot{X}_{\theta_0},\dot{X}_{\theta_1}]_r = \dot{X}_{(1-r) \theta_0 + r \theta_1}
\end{equation}
for any $\theta_0,\theta_1 \in \R$ and $r \in (0,1),$ see \cite{Trie} and \cite[Proposition 2.2]{KaKW}.
The connection between the complex interpolation scale $\dot{X}_{\theta}$ and the $\HI$ calculus has been studied e.g. in \cite{KaKW, WeSurvey}.

\subsection{Generalized Square Functions}\label{Subsec Prelims Rad}

A classical theorem of Marcinkiewicz and Zygmund states that for elements $x_1,\ldots,x_n \in L^p(U,\mu)$ we can express ``square sums'' in terms of random sums
\begin{equation}\label{Equ PL equivalence} \left\| \left( \sum_{j=1}^n |x_j(\cdot)|^2 \right)^{\frac12} \right\|_{L^p(U)}
\cong \left( \E \| \sum_{j=1}^n \epsilon_j x_j \|_{L^p(u)}^q \right)^{\frac1q}
\cong \left( \E \| \sum_{j=1}^n \gamma_j x_j \|_{L^p(u)}^q \right)^{\frac1q}\end{equation}
with constants only depending on $p,q \in [1,\infty).$
Here $(\epsilon_j)_j$ is a sequence of independent Bernoulli random variables (with $P(\epsilon_j = 1) = P(\epsilon_j = -1) =\frac12$) and $(\gamma_j)_j$ is a sequence of independent standard Gaussian random variables.
Following \cite{Bou} it has become standard by now to replace square functions in the theory of Banach space valued function spaces by such random sums (see e.g. \cite{KuWe}).
Note however that Bernoulli sums and Gaussian sums for $x_1,\ldots,x_n$ in a Banach space $X$ are only equivalent if $X$ has finite cotype (see \cite[p.~218]{DiJT} for details).

A Banach space version of continuous square functions $\| \left( \int_I |f(t)|^2 dt\right)^{\frac12} \|_{L^p(U)},\: I \subset \R$ for functions $f : I \to X$ was developped in \cite{KaW2}.
Assume that $x' \circ f \in L^2(I,dt)$ for all $x' \in X'$ and define an operator $u_f : L^2(I) \to X$ by
\[ \spr{u_f h}{x'} = \int_I \spr{f(t)}{x'} h(t) dt,\: h \in L^2(I). \]
Denote by $\|f\|_{\gamma(I,X)}$ the $\gamma$-radonifying norm of the operator $u_f:$
\[ \| u_f \|_{\gamma(L^2(I),X)} := \left( \E \| \sum_k \gamma_k u_f(e_k) \|_X^2 \right)^{\frac12} \]
where $(e_k)$ is an orthonormal basis of $L^2(I).$
We use Gaussian variables here, so that the norm $\|u_f\|$ is independent of the specific choice of the orthonormal basis $(e_k).$
Note that for $f : I \to L^p(U,\mu)$ we have
\begin{align*} \| \left( \int_I |f(t)|^2 dt \right)^{\frac12} \|_{L^p(U)} & = \| \left( \sum_k | \int_I f(t) e_k(t) dt |^2 \right)^{\frac12} \|_{L^p(U)} \cong \left( \E \left\| \sum_k \gamma_k \int_I f(t) e_k(t) dt \right\|^2 \right)^{\frac12} \\
&  = \|u_f\|_\gamma = \|f\|_{\gamma(I,L^p(U))} 
\end{align*}
as intended.
The completion of the space of functions $f : I \to X$ with $\|f\|_{\gamma(I,X)} < \infty$ is the space of radonifying operators $\gamma(L^2(I),X).$
For this fact, the following properties and further extensions see \cite{DiJT} and \cite{vN}.

\begin{lem}\label{Lem Technical gamma spaces}
Let $(\Omega_k,\mu_k)$ be $\sigma$-finite measure spaces $(k = 1,2).$
\begin{enumerate}
\item For $f \in \gamma(\Omega_1,X)$ and $g \in \gamma(\Omega_1,X'),$ we have
\[ \int_{\Omega_1} \left| \spr{f(t)}{g(t)} \right| dt \leq \|f\|_{\gamma(\Omega_1,X)} \|g\|_{\gamma(\Omega_1,X')}. \]
\item Let $g : \: \Omega_1 \otimes \Omega_2 \to X$ be weakly measurable and assume that for any $x' \in X',$ we have
$\int_{\Omega_1} \left( \int_{\Omega_2} \left| \spr{g(t,s)}{x'} \right| ds \right)^2 dt < \infty.$
Then
$\int_{\Omega_2} g(\cdot,s) ds \in \gamma(\Omega_1,X)$ and $\|\int_{\Omega_2} g(\cdot,s) ds \| \leq \int_{\Omega_2} \|g(\cdot,s)\|_{\gamma(\Omega_1,X)} ds$
hold as soon as the right most expression is finite.
\end{enumerate}
\end{lem}

Let $\tau$ be a subset of $B(X).$
We say that $\tau$ is $R$-bounded if there exists a $C < \infty$ such that
\[ \E \bignorm{ \sum_{k=1}^n \epsilon_k T_k x_k } \leq C \E \bignorm{ \sum_{k=1}^n \epsilon_k x_k } \]
for any $n \in \N,$ $T_1,\ldots, T_n \in \tau$ and $x_1,\ldots,x_n \in X.$
The smallest admissible constant $C$ is denoted by $R(\tau).$
We remark that one always has $R(\tau) \geq \sup_{T \in \tau} \|T\|$ and equality holds if $X$ is a Hilbert space.

\section{The Mihlin functional calculus}

We will use spectral multiplier theorems for the following Mihlin classes of functions for $\alpha > 0.$
\[\Ma = \{ f : \R_+ \to \C:\: f_e \in \Ba\},\]
equipped with the norm $\|f\|_{\Ma} = \|f_e\|_{\Ba}.$ 
Here and later we write
\[f_e : J \to \C,\,z \mapsto f(e^z)\]
for a function $f : I \to \C$ such that $I \subset \C \backslash (-\infty,0]$ and $J = \{ z \in \C : \: | \Im z | < \pi,\:e^z \in I \}$ and similarly $f_2 = f(2^{(\cdot)}).$
The space $\Ma$ coincides with the space $\Lambda^\alpha_{\infty, 1}(\R_+)$ in \cite[p. 73]{CDMY}.
The name ``Mihlin class'' is justified by the following facts.
The Mihlin condition for a $\beta$-times differentiable function $f: \R_+ \to \C$ is
\begin{equation}\label{Equ Classical Mihlin condition}
\sup_{t > 0, k = 0,\ldots,\beta} |t^k f^{(k)}(t)| < \infty
\end{equation}
\cite[(1)]{Duon}.
If $f$ satisfies \eqref{Equ Classical Mihlin condition}, then $f \in \Ma$ for $\alpha < \beta$ \cite[p. 73]{CDMY}.
Conversely, if $f \in \Ma,$ then $f$ satisfies \eqref{Equ Classical Mihlin condition} for $\alpha \geq \beta.$
The proof of this can be found in \cite[Theorem 3.1]{GaMi}, where also the case $\beta \not\in \N$ is considered.

We have the following elementary properties of Mihlin spaces.
Its proof may be found in \cite[Propositions 4.8 and 4.9]{Kr}.

\begin{prop}\label{Prop Elementary Mih Hor}~
\begin{enumerate}
\item The space $\Ma$ is a Banach algebra.
\item Let $m,n \in \N_0$ and $\alpha,\beta > 0$ such that $m > \beta > \alpha > n.$
Then
\[C^m_b \hookrightarrow \Bes^\beta_{\infty,\infty} \hookrightarrow \Ba \hookrightarrow \Bes^\alpha_{\infty,\infty} \hookrightarrow C^n_b.\]
\end{enumerate}
\end{prop}

We will define the $\Ma$ functional calculus for a $0$-sectorial operator $A$ 
in terms of the holomorphic functional calculus.
The following lemma from \cite[Lemma 4.15]{Kr} will be useful.

\begin{lem}\label{Lem HI dense in diverse spaces}~
Let $f \in \Ma$ and $(\Fdyad_n)_n$ a dyadic partition of unity on $\R.$
\begin{enumerate}
\item The series $f = \sum_{n \in \Z} (f_e \ast \check{\Fdyad_n}) \circ \log$ converges in $\Ma.$
Note that $(f_e \ast \check{\Fdyad_n})\circ \log$ belongs to 
\begin{equation}\label{Equ HI dense in Ba}
\bigcap_{0 < \omega < \pi} \HI(\Sigma_\omega) \cap \Ma,
\end{equation}
so that \eqref{Equ HI dense in Ba} is dense in $\Ma.$
\item Let $\psi \in C^\infty_c(\R)$ such that $\psi(t) = 1$ for $|t|\leq 1$ and $\psi(t) = 0$ for $|t| \geq 2.$
Further let $\psi_n = \psi(2^{-n}\cdot).$
Then $(f_e \ast \check{\psi_n})\circ \log$ converges to $f$ in $\Ma.$
\end{enumerate}
\end{lem}

Lemma \ref{Lem HI dense in diverse spaces} enables to base the $\Ma$ calculus on the $\HI$ calculus.

\begin{defi}\label{Def Line calculi}
Let $A$ be a 0-sectorial operator and $\alpha > 0.$
We say that $A$ has a (bounded) $\Ma$ calculus if there exists a constant $C > 0$ such that
\[
\|f(A)\| \leq C \|f\|_{\Ma}\quad (f \in \bigcap_{0 <\omega < \pi} \HI(\Sigma_\omega) \cap \Ma).
\]
In this case, by density of $\bigcap_{0<\omega < \pi} \HI(\Sigma_\omega) \cap \Ma$ in $\Ma,$ the algebra homomorphism $u : \bigcap_{0<\omega < \pi} \HI(\Sigma_\omega) \cap \Ma \to B(X)$
given by $u(f) = f(A)$ can be continuously extended in a unique way to a bounded algebra homomorphism
\[u: \Ma \to B(X),\,f \mapsto u(f).\]
We write again $f(A) = u(f)$ for any $f \in \Ma.$
\end{defi}

We recall that (cf. \cite{CDMY}) $A$ has a $\Ma$ calculus if and only if $A$ has a $\HI(\Sigma_\omega)$ calculus for all $\omega > 0$ and there is a constant $C$ such that
\[ \| f(A) \| \leq C \omega^{-\alpha} \| f \|_{\HI(\Sigma_\omega)} \]
for all $\omega > 0.$
The following convergence property extends the well-known Convergence Lemma for the $\HI$ calculus \cite[Lemma 2.1]{CDMY}.

\begin{prop}\label{Prop Convergence Bes}
Let $A$ be a $0$-sectorial operator with bounded $\Ma$ calculus for some $\alpha >  0.$
Then the following convergence property holds.
Let $\beta > \alpha$ and $(f_n)_{n \in \N}$ be a sequence such that $f_{n,e}$ belongs to $\Bbii$ with
\begin{enumerate}
\item[(a)] $\sup_{n \in \N} \|f_{n,e}\|_{\Bbii} < \infty,$
\item[(b)] $f_n(t) \to f(t)$ pointwise on $\R_+$ for some function $f.$
\end{enumerate}
Then
\begin{enumerate}
\item $f_e \in \Bbii,$
\item $f_n(A) x \to f(A) x$ for all $x \in X.$
\end{enumerate}
In particular, if $(\ddyad_n)_{n \in \Z}$ is a dyadic partition of unity on $\R_+$, then for any $x \in X,$
\begin{equation}\label{Equ Convergence Bes}
x = \sum_{n \in \Z} \ddyad_n(A)x \quad (\text{convergence in }X).
\end{equation}
\end{prop}

\begin{proof}
By \cite[Theorem 2.5.12]{Triea}, the norm on $\Bes^{\beta}_{\infty,\infty}$ is equivalent to the following norm:
\begin{equation}\label{Equ Bbii norm}
\|g\|_{L^\infty(\R)} + \sup_{x,t \in \R,t \neq 0} |t|^{-\beta} |\Delta_t^M g(x)|
\end{equation}
for a fixed $M \in \N$ such that $M > \beta.$
Here, $\Delta_t^M$ is the iterated difference defined recursively by $\Delta_t^Mg(x) = \Delta_t^1(\Delta_t^{(M-1)}g)(x)$ and $\Delta_t^1 g(x) = g(x + t) - g(x).$
More precisely, by \cite[Remark 3]{Triea}, a $g \in L^\infty(\R)$ belongs to $\Bbii$ if and only if the above expression is finite.
We have
\[\sup_{x,t \in \R,t \neq 0} |t|^{-\beta} |\Delta_t^M f_e(x)| = \sup_{x,t} |t|^{-\beta} \lim_n |\Delta_t^M f_{n,e}(x)|
\leq \sup_{x,t,n} |t|^{-\beta} |\Delta_t^M f_{n,e}(x)| \leq \sup_n \|f_{n,e}\|_{\Bes^\beta_{\infty,\infty}} < \infty\]
and similarly $\|f\|_\infty \leq \sup_n \|f_n\|_\infty.$
Therefore, assertion (1) of the proposition follows.\\

Let $(\Fdyad_k)_{k \in \Z}$ be a dyadic partition of unity on $\R.$
Note that by the boundedness of the $\Ma$ calculus and Lemma \ref{Lem HI dense in diverse spaces},
$f_n(A) = \sum_k (f_{n,e} \ast \check{\Fdyad_k})\circ\log(A).$
We first show the stated convergence for each summand and claim that for any $x\in X$ and fixed $k\in \Z,$
\begin{equation}\label{Equ 6}
(f_{n,e} \ast \check{\Fdyad_k})\circ\log (A) x \to (f_e \ast \check{\Fdyad_k})\circ\log(A)x.
\end{equation}
Indeed, this follows from the well-known Convergence Lemma of the $\HI$ calculus \cite[Lemma 2.1]{CDMY}.
Fix some angle $\omega > 0.$
Firstly,
\[ \|(f_{n,e} \ast \check{\Fdyad_k})\circ\log \|_{\infty,\omega} \leq \|f_n\|_{L^\infty(\R_+)} \sup_{|\theta| < \omega} \| \check{\Fdyad_k}(i \theta - \cdot)\|_{L^1(\R)} \leq C.\]
Secondly, for any $z \in \Sigma_\omega,$
\[ f_{n,e} \ast \check{\Fdyad_k}(\log z) = \int_\R f_{n,e}(s) \check{\Fdyad_k}(\log z-s) ds \to \int_\R f(s) \check{\Fdyad_k}(\log z-s) ds = f \ast \check{\Fdyad_k}(\log z)\]
by dominated convergence, and \eqref{Equ 6} follows.
For $n \in \N$ and $k\in \Z,$ put $x_{n,k} = (f_{n,e} \ast \check{\Fdyad_k})\circ\log(A)x,$ where $x\in X$ is fixed.
For any $N \in \N,$
\begin{align}
\|f_{n}(A)x - f(A)x\| & \leq \sum_{k\in \Z} \|(f_{n,e}\ast\check{\Fdyad_k})\circ\log(A)x -(f_e\ast\check{\Fdyad_k})\circ\log(A)x\|
\nonumber \\
& \leq \sum_{|k| \leq N} \|x_{n,k} - (f_e\ast\check{\Fdyad_k})\circ\log(A)x\| + \sum_{|k| > N} \left(\|x_{n,k}\| + \|(f_e\ast\check{\Fdyad_k})\circ\log(A)x\|\right)
\nonumber.
\end{align}
By \eqref{Equ 6}, we only have to show that
\begin{equation}\label{Equ 5}
\lim_{N \to \infty} \sup_n \sum_{|k| > N} \|x_{n,k}\| = 0.
\end{equation}
Fix some $\gamma \in (\alpha,\beta).$
\begin{align*}
\|x_{n,k}\| & \lesssim \|f_{n,e}\ast\check{\Fdyad_k}\|_{\Ba}\|x\| \lesssim \|f_{n,e}\ast\check{\Fdyad_k}\|_{\Bes^\gamma_{\infty,\infty}} \|x\|
= \sup_{l \in \Z} 2^{|l|\gamma} \| f_{n,e} \ast \check{\Fdyad_k} \ast \check{\Fdyad_l} \|_\infty \|x\| \\
& = \sup_{|l-k|\leq 1} 2^{|l|\gamma} \|f_{n,e} \ast \check{\Fdyad_k} \ast \check{\Fdyad_l} \|_\infty \|x\|
\lesssim 2^{|k|\gamma} \|f_{n,e} \ast \check{\Fdyad_k}\|_\infty \|x\|.
\end{align*}
By assumption of the proposition,
$\sup_n \|f_{n,e}\|_{\Bes^\beta_{\infty,\infty}} = \sup_{n,k} 2^{|k|\beta} \|f_{n,e} \ast \check{\Fdyad_k}\|_\infty < \infty.$
Therefore,
$\|x_{n,k}\| \lesssim 2^{|k|\gamma} 2^{-|k|\beta}\|x\|,$
and \eqref{Equ 5} follows.
To show \eqref{Equ Convergence Bes}, set $f_k = \sum_{n = -k}^k \ddyad_n.$
Then the sequence $(f_k)_{k \in \N}$ satisfies (a) and (b) with limit $f = 1.$
Indeed, the pointwise convergence is clear.
The uniform boundedness $\sup_{k \geq 1} \| f_{k,e} \|_{\Bbii} < \infty$ can be shown with Lemma \ref{Lem Overlapping} in Section \ref{Sec 3 Spectral Decomposition} below.
\end{proof}

Let $(\ddyad_n)_n$ be a homogeneous dyadic partition of unity on $\R_+$ and
\begin{equation}\label{Equ D}
D_A = \{ x \in X : \: \exists \, N \in \N :\: \ddyad_n(A)x  = 0 \quad (|n| \geq N) \}.
\end{equation}
Then $D_A$ is a dense subset of $X$ provided that $A$ has a bounded $\Ma$ calculus.
Indeed, for any $x \in X$ let $x_N = \sum_{k=-N}^{N} \ddyad_k(A)x.$
Then for $|n| \geq N + 1,\,\ddyad_n(A) x_N = \sum_{k=-N}^N (\ddyad_n \ddyad_k)(A)x = 0,$ so that $x_N$ belongs to $D_A.$
On the other hand, by \eqref{Equ Convergence Bes}, $x_N$ converges to $x$ for $N \to \infty.$
Clearly, $D_A$ is independent of the choice of $(\ddyad_n)_n.$
We call $D_A$ the calculus core of $A.$

Since a $\Ma$ calculus does require a bounded $\HI$ calculus, we would like to consider also a weaker calculus, which is much easier to check than a $\Ma$ calculus.

\begin{defi}
We define
\[ \Ma_1 = \{ f \in \Ma :\: \|f\|_{\Ma_1} = \sum_{n \in \Z} \| f \ddyad_n \|_{\Ma} < \infty \},\]
where $(\ddyad_n)_n$ is a fixed dyadic partition of unity on $\R_+.$
It is clear that this definition does not depend on the particular choice of $(\ddyad_n)_n,$ that $\Ma_1 \subset \Ma$ and that $\Ma_1$ is a Banach algebra.
\end{defi}

\begin{lem}\label{Lem HI0 dense in l1Ma}
Let $\alpha > 0$ and $\omega \in (0,\pi).$
The space $\HI_0(\Sigma_\omega)$ is contained in $\Ma_1$ and is dense in it.
\end{lem}

\begin{proof}
We first show that $\HI_0(\Sigma_\omega) \subset \Ma_1.$
Let $C , \epsilon > 0$ such that $|f(\lambda)| \leq C (1 +|\log \lambda|)^{-1-\epsilon}$ for $\lambda \in \Sigma_\omega.$
For $n \in \Z$ fixed, $t \in [2^{n-1},2^{n+1}]$ and any $k \in \N,$ we have by the Cauchy integral formula $|t^k f^{(k)}(t)| \leq C_{k,\omega} \sup_{\lambda \in \Sigma_\omega,\,|\lambda| \in [2^{n-2},2^{n+2}]} |f(\lambda)| \leq C_{k,\omega}' (1+|\log 2^{n}|)^{-1-\epsilon}.$
Thus by \eqref{Equ Classical Mihlin condition}, $\|f \ddyad_n\|_{\Ma} \lesssim (1 + |n|)^{-1-\epsilon},$ the right hand side being clearly summable over $n \in \Z.$

Now for the density statement.
To this end, let $\psi \in C^\infty_c(\R)$ with $\supp \psi \subseteq [-2,2]$ and $\psi(t) = 1$ for $t \in [-1,1].$
Further let for $n \in \N,$ $\psi_n = \psi(2^{-n} \cdot).$
Let $f \in \Ma_1$ which shall be approximated by $\HI_0(\Sigma_\omega)$ functions.
Since the span of compactly supported functions is dense in $\Ma_1,$ we can assume that $f$ has support say in $[2^{l_0 - 1},2^{l_0 + 1}].$
Let $f_n = (f_e \ast \check{\psi_n})\circ \log,$ which belongs to $\HI(\Sigma_\omega),$ since $\check{\psi_n}$ is an entire function.
It also belongs to $\HI_0(\Sigma_\omega),$ since $\sup_{|\theta| < \omega}|\check{\psi_n}(\cdot + i \theta)|$ decreases rapidly and $f_e$ has compact support.
We claim that $f_n \to f$ in $\Ma_1$ as $n \to \infty.$
Indeed, we have
\[ \sum_{l \in \Z} \|(f-f_n)\ddyad_l\|_{\Ma} \lesssim \sum_{|l| \leq L} \| f - f_n \|_{\Ma} + \sum_{|l| > L} \|f \ddyad_l\|_{\Ma} + \sum_{|l| > L}\|f_n \ddyad_l \|_{\Ma}.\]
The terms in the first sum converge to $0$ according to Lemma \ref{Lem HI dense in diverse spaces}.
The second sum vanishes for $L > |l_0| + 1,$ so it remains to show that the third sum converges to $0$ as $L \to \infty,$ uniformly in $n \in \N.$
In the following calculation, we use the norm description of the Besov space $\Bes^\alpha_{\infty,1}$, see \cite[p.~110]{Triea}
\[ \| g \|_{\Bes^\alpha_{\infty,1}} \cong \| g \|_{L^\infty(\R)} + \int_{-1}^1 |h|^{-\alpha} \sup_{x \in \R} | \Delta^M_h g(x) | \frac{dh}{|h|} ,\]
where $M > \alpha$ and $\Delta^M_h g(x)$ is the iterated difference as already used in \eqref{Equ Bbii norm}.
Note that it commutes with convolutions, i.e. $\Delta^M_h (g \ast k)(x)= (\Delta^M_h g) \ast k (x),$ and for products, we have
$\Delta^M_h(gk) = \sum_{m=0}^M {M \choose m} \Delta^m_h g \Delta^{M-m}_h \tau_{mh} k,$ where $\tau_{y} k (x) = k(x+y).$
Then
\[ \sum_{|l| > L} \|f_n \ddyad_l\|_{\Ma} \cong \sum_{|l| > L} \| (f_e \ast \check{\psi_n})\cdot \ddyad_{l,e} \|_\infty + \int_{-1}^1 |h|^{-\alpha} \sup_{x \in \R} | \Delta^M_h (f_e \ast \check{\psi_n} \cdot \ddyad_{l,e})(x)| \frac{dh}{|h|}.\]
We estimate
\begin{align*} \sum_{|l| > L} \| f_e \ast \check{\psi_n} \cdot \ddyad_{l,e} \|_\infty & \lesssim \sum_{|l| > L} \| f_e \ast \check{\psi_n} \|_{L^\infty([\log(2) (l-1),\log(2) (l+1)])} \\
& \leq \sum_{|l| > L} \|f_e\|_{\infty} \|\check{\psi_n}\|_{L^1([\log(2)(l+l_0-2),\log(2)(l+l_0+2)])}.
\end{align*}
Note that $\check{\psi_n} = 2^n \check{\psi}(2^n \cdot)$ and that $\check{\psi}$ is rapidly decreasing.
Thus, the above last sum is finite and converges to $0$ as $L \to \infty,$ uniformly in $n.$
Finally, we estimate
\begin{align*}
\int_{-1}^1 & |h|^{-\alpha} \sup_{x \in \R} |\Delta^M_h(f_e \ast \check{\psi_n} \cdot \ddyad_{l,e})(x)| \frac{dh}{|h|}
\lesssim \int_{-1}^1 |h|^{-\alpha} \sum_{m=0}^M \sup_{|x-\log(2) l| < \log(2)} |\Delta^m_h(f_e \ast \check{\psi_n})(x)| \: \|\Delta_h^{M-m}\ddyad_{l,e}\|_\infty \frac{dh}{|h|} \\
& = \int_{-1}^1 |h|^{-\alpha} \sum_{m=0}^M \sup_{|x-\log(2) l| < \log(2)} |(\Delta^m_h f_e) \ast \check{\psi_n}(x)| \: \|\Delta_h^{M-m}\ddyad_{l,e}\|_\infty \frac{dh}{|h|} \\
& \leq \int_{-1}^1 |h|^{-\alpha} \sum_{m=0}^M \| \Delta^m_h f_e \|_{L^\infty(\R)}  \|\check{\psi_n}\|_{L^1([\log(2)(l+l_0-2),\log(2)(l+l_0+2)])} \: \|\Delta_h^{M-m}\ddyad_{l,e}\|_\infty \frac{dh}{|h|} \\
& \lesssim \| f_e \|_{\Bes^\alpha_{\infty,1}}\|\check{\psi_n}\|_{L^1([\log(2)(l+l_0-2),\log(2)(l+l_0+2)])} \|\ddyad_{l,e}\|_{\Bes^\alpha_{\infty,1}},
\end{align*}
which is again summable in $|l| > L,$ and converges to $0$ as $L \to \infty,$ uniformly in $n \in \N.$
\end{proof}

The following observation will be useful later on.

\begin{lem}
Let $\alpha > 0,\:f \in \Ma_1,\: \phi \in C^\infty_c(0,\infty)$ and assume that $|f(\lambda)|> \beta >0$ for $\lambda \in \supp (\phi).$
Then $\phi \cdot f^{-1}$ belongs again to $\Ma_1.$
\end{lem}

\begin{proof}
Replacing $f$ by $f_e = f \circ \exp$ and using the compact support of $\phi,$ we are reduced to show that if $f \in \Ba$ and $\phi \in C^\infty_c(\R)$ such that $|f(x)|>\beta > 0$ for $x \in \supp(\phi),$ then $\phi \cdot f^{-1}$ belongs to $\Ba.$
We use the norm description from \cite[p.~110]{Triea}
\[ \|g\|_{\Ba} \cong \|g\|_{L^\infty(\R)} + \int_{-1}^1 |h|^{-\alpha} \|\Delta^M_h g\|_{L^\infty(\R)} \frac{dh}{|h|},\]
where $M > \alpha$ and $\Delta^M_h$ stands for the iterated difference operator, defined by $\Delta^1_hg(x) = g(x+h)-g(x)$ and recursively, $\Delta^{M+1}_hg(x) = \Delta^1_h ( \Delta^M_g)(x).$
Using the product formula $\Delta^M_h(gk) = \sum_{m=0}^M {M \choose m} \Delta^m_hg \Delta^{M-m}_h \tau_{mh}k,$ where $\tau_yk(x) = k(x+y),$ we can estimate
\[ \| \phi \cdot f^{-1} \|_{\Ba} \lesssim \|\phi \cdot f^{-1}\|_\infty + \sum_{m=0}^M \int_{-1}^1 |h|^{-\alpha_1(m)} \|\Delta^m_h \phi\|_\infty |h|^{-\alpha_2(m)} \|\Delta^{M-m}_h \tilde{f}^{-1}\|_{\infty} \frac{dh}{|h|},\]
where $\alpha = \alpha_1(m) + \alpha_2(m)$ is a decomposition such that $\alpha_1(m) < m$ and $\alpha_2(m) < M - m,$ and $\tilde{f}^{-1}$ is equal to $f^{-1}$ on $\supp(\phi)$ and smoothly and boundedly from above and below continued beyond $\supp(\phi).$
Since $\phi \in C^\infty_c(\R) \subset \Bes^{\alpha_1(m)}_{\infty,\infty},$ we have $\sup_{|h|<1} |h|^{-\alpha_1(m)} \| \Delta^m_h \phi \|_\infty < \infty,$ so that now it suffices to show that 
\[ \int_{-1}^1 |h|^{-\alpha} \|\Delta^M_h f^{-1}\|_\infty \frac{dh}{|h|} < \infty,\]
where we moreover assumed that $\tilde{f}^{-1} = f^{-1}$ and $M,\alpha$ takes over the role of $M-m$ and $\alpha_2(m).$
We now claim that
\[ \Delta^M_h f^{-1}(x) = \sum_{k=2}^{M+1} \sum_{l=1}^{L_{k,M}} c_{k,l,M} \frac{X_1 \cdot \ldots \cdot X_{p_{k,l,M}}}{f(x)f(x+h)f(x+2h)\cdot \ldots \cdot f(x + (k-1)h)} ,\]
where the $X_p$ are of the form $X_p =\Delta^{m_p}_{h} \tau_{lh}f$ such that each time $\sum_{p=1}^{p_{k,l,M}} m_p = M.$
This claim can be shown by induction, using the Leibniz's type rule $\Delta^1_h (f_1 \cdot \ldots \cdot f_N) = \Delta^1_hf_1(x) \cdot f_2(x+h)\cdot \ldots \cdot f_N(x+h) + f_1(x) \Delta^1_hf_2(x) f_3(x+h) \cdot \ldots \cdot f_N(x+h) + \ldots + f_1(x) \cdot \ldots f_{N-1}(x) \Delta^1_h f_N(x),$ and the fact that $\Delta^1_h \frac{1}{f(x)f(x+h)\cdot \ldots \cdot f(x+(k-1)h)} = - \frac{\Delta^1_{h}\tau_{(k-1)h}f(x) + \Delta^1_h\tau_{(k-2)h}f(x) + \ldots + \Delta^1_hf(x)}{f(x)f(x+h)\cdot \ldots \cdot f(x+kh)}.$
This gives the estimate
\[ \int_{-1}^1 |h|^{-\alpha} \| \Delta^M_h f^{-1} \|_\infty \frac{dh}{|h|}
\lesssim \sum_{k=2}^{M+1} \sum_{l=1}^{L_{k,m}} \int_{-1}^1 |h|^{-\alpha} \|f^{-1}\|_\infty^k \prod_{p=1}^{p_{k,l,M}} \|\Delta^{m_p}_{h} f\|_\infty \frac{dh}{|h|}. \]
Now argue as before: decompose $\alpha = \alpha(1) + \ldots + \alpha(p_{k,l,M})$ with $\alpha(p) < m_p$ (which is possible since $\sum_{p=1}^{p_{k,l,M}} m_p = M > \alpha$), and use the fact that $\Ba \hookrightarrow \Bes^{\alpha(p)}_{\infty,\infty}$ for $\alpha(p) \leq \alpha,$ so that
$\sup_{h \neq 0} |h|^{- \alpha(p)} \|\Delta^{m_p}_{h} f\|_\infty < \infty$ for $p=1,\ldots,p_{k,l,M}-1.$
Use this estimate for all factors $|h|^{-\alpha(p)} \|\Delta^{m_p}_h f\|_\infty$ except the last, for which we use the $\Bes^{\alpha(p_{k,l,M})}_{\infty,1}$ norm, to deduce finally that
\[ \int_{-1}^1 |h|^{-\alpha} \prod_{p=1}^{p_{k,l,M}} \|\Delta_{h}^{m_p} f\|_\infty \frac{dh}{|h|} < \infty .\]
\end{proof}

\begin{defi}\label{Def ell1Ma calculus}
Let $A$ be a $0$-sectorial operator and $\alpha > 0.$
We say that $A$ has an $\Ma_1$ calculus if there is a constant $C > 0$ such that $\|f(A)\| \leq C \|f\|_{\Ma_1}$ for any $f \in \HI_0(\Sigma_\omega)$ and for some $\omega \in (0,\pi).$
In this case, by the density proved in Lemma \ref{Lem HI0 dense in l1Ma} above, we can extend the definition of $f(A)$ to all $f \in \Ma_1$
and have in particular $f(A) g(A) = (fg)(A)$ for any $f,g \in \Ma_1.$
\end{defi}

The set $\Ma_1$ is too small to contain many interesting, in particular singular spectral multipliers.
However the boundedness of an $\Ma_1$ calculus follows directly from common norm estimates for sectorial operators without additional information on kernel or square function estimates.
The following proposition gives a sufficient condition in terms of resolvents.

\begin{prop}\label{Prop Resolvent Ma0}
Let $A$ be a $0$-sectorial operator and $\alpha > 0.$
If
\[\sup \{ \| \lambda R(\lambda,A) \| :\: |\arg \lambda| = \omega \} \lesssim \omega^{-\alpha}\]
for $\omega \in (0,\pi),$ then $A$ has a bounded $\Ma_1$ calculus.
If 
\[R( \lambda R(\lambda,A) :\: |\arg \lambda| = \omega ) \lesssim \omega^{-\alpha}\]
for $\omega \in (0,\pi),$ then $A$ has an $R$-bounded $\Ma_1$ calculus, i.e. $\{ f(A) :\: f\in \Ma_1,\: \|f\|_{\Ma_1} \leq 1 \}$ is $R$-bounded.
\end{prop}

\begin{proof}
Let $(\Fdyad_m)_{m \in \Z}$ be a dyadic partition of unity on $\R.$
Then for $m \geq 1,$ we have $\Fdyad_m\ck(x+ i 2^{-m}) = [\Fdyad_1(2^{-m+1}\cdot) \exp(2^{-m} \cdot)]\ck(x).$
Since $\Fdyad_1(2 \cdot) \exp(\cdot)$ belongs to $C^\infty_c(\R),$ $[\Fdyad_1(2 \cdot) \exp(\cdot)]\ck$ is a rapidly decreasing function, so that in particular for $N > \alpha + 1,$
$|[\Fdyad_1(2\cdot) \exp(\cdot)]\ck(\xi)| \leq C ( 1 + |\xi| )^{-N}.$
Thus, $|\Fdyad_m\ck(x+ i2^{-m})| = 2^m |[\Fdyad_1(2\cdot) \exp(\cdot)]\ck(2^m x)| \leq C 2^m (1 + |2^m x|)^{-N}.$
Similarly, for $m \leq -1,$ one obtains $|\Fdyad_m\ck(x+ i2^{-|m|})| \leq C 2^{|m|} ( 1 + | 2^{|m|} x |)^{-N}.$
Let now $f \in \Ma_1$ with support in $[\frac12,2].$
Put $\rho_{m,e} = f_e \ast \Fdyad_m\ck.$
Then 
\[|\rho_{m,e}(x + i2^{-|m|})| \leq 2 \log(2) \|f\|_\infty \sup_{|y-x| \leq \log(2)} |\Fdyad_m\ck(y+i2^{-|m|})| \leq C \|f\|_\infty 2^{-(N-1)|m|} |x|^{-N}.\]
Moreover, by the Paley-Wiener theorem combined with the fact that $\hat{\rho}_{m,e}$ is supported in $[-2^{|m|+1},2^{|m|+1}],$
\[|\rho_{m,e}(x + i2^{-|m|})| \leq e^2 \sup_{y\in\R} |\rho_{m,e}(y)|.\]
Thus, for $\|x\| \leq 1,$
\begin{align*}
\|f(A)x\| & \leq \sum_{m \in \Z} \|\rho_m(A)\| \leq \frac{1}{2\pi} \sum_{m \in \Z} \int_{|\arg \lambda| = 2^{-|m|}} |\rho_m(\lambda)| \cdot \| \lambda R(\lambda,A)\| \left| \frac{d\lambda}{\lambda} \right| \\
& = \frac{1}{2\pi} \sum_{m \in \Z} \int_{|\arg\lambda| = 2^{-|m|},\:|\lambda| \in [\frac12,2]} |\rho_m(\lambda)| \cdot \|\lambda R(\lambda,A)\| \left| \frac{d\lambda}{\lambda} \right| \\
& + \frac{1}{2\pi} \sum_{m \in \Z} \int_{|\arg\lambda| = 2^{-|m|},\:|\lambda| \not\in [\frac12,2]} |\rho_m(\lambda)| \cdot \|\lambda R(\lambda,A)\| \left| \frac{d\lambda}{\lambda} \right| \\
& \lesssim \sum_{m \in \Z} \|\rho_{m,e}\|_{L^\infty(\R)} \cdot 2^{|m|\alpha} + \sum_{m \in \Z} \int_{|x| \geq \log(2)} 2^{-(N-1)|m|} |x|^{-N} 2^{|m|\alpha} dx \|f\|_{\infty} \\
& \lesssim \|f\|_{\Ma} + \|f\|_\infty \cong \|f\|_{\Ma},
\end{align*}
where we have used in the penultimate line that $N-1 > \alpha$ and $N > 1.$
If $g$ is a function in $\Ma_1$ with support in $[\frac12 t,2t],$ put $f(x) = g(tx)$ so that $\supp f \subset [\frac12,2].$
Let $\rho_m$ be as above.
Then
\begin{align*}
\|g(A)x\| & \leq \sum_{m\in\Z} \|\rho_m(t^{-1}A)\| \leq \frac{1}{2\pi} \sum_{m\in\Z} \int_{|\arg \lambda| = 2^{-|m|} } |\rho_m(\lambda)| \cdot \|t \lambda R(t \lambda,A)\| \left|\frac{d \lambda}{\lambda}\right| \\
& \lesssim \|f\|_{\Ma} = \|g\|_{\Ma}.
\end{align*}
For a general $f \in \Ma_1,$ we have by the above 
\[ \|f(A)\| \leq \sum_{n \in \Z} \| (f\ddyad_n)(A) \| \lesssim \sum_{n \in \Z} \|f\ddyad_n\|_{\Ma} = \|f\|_{\Ma_1}.\]
Now the first claim of the proposition follows.
For the second claim, let $f_1,\ldots,f_N \in \Ma_1$ with supports in an interval of the form $[s_k/2,2_k]$ and $\|f_k\|_{\Ma} \leq 1.$
Put $\rho_{m,e}^k = f_{k,e}(\cdot - t_k) \ast \Fdyad_m\ck$ with $t_k$ such that $\supp f_{k,e}(\cdot-t_k) \subset [-\log(2),\log(2)].$
Then
\[
f_k(A) = \sum_{m \in \Z} \rho_m^k(e^{t_k} A) = \frac{1}{2\pi} \sum_{m \in \Z} \int_{|\arg \lambda| = 2^{-|m|}} \rho_m^k(\lambda) e^{-t_k} \lambda R(e^{-t_k} \lambda,A) \frac{d\lambda}{\lambda}.
\]
The claim follows now as the first claim, with the $R$-bound integral lemma \cite[2.14 Corollary]{KuWe} in place of norm estimates of $\lambda R(\lambda,A).$
\end{proof}

\begin{rem}\label{rem N_T}
If the $0$-sectorial operator $A$ satisfies the norm estimate
\[ \left\{ \left(\frac{|\Re z|}{|z|}\right)^\beta e^{-zA} :\: \Re z > 0 \right\} \text{ is bounded} \]
(or $R$-bounded) then $A$ has a ($R$-bounded) $\Mih^{\beta + 1}_1$ calculus.
Indeed the assumption of Proposition \ref{Prop Resolvent Ma0} is satisfied with $\alpha = \beta + 1$ since for $z = r e^{-i\theta}$ and $s > 0$
\[ (e^{i\theta s} - zA)^{-1} = \int_0^\infty \exp(-e^{i\theta} st - tz A) dt, \]
or $e^{i\theta} s (s e^{i 2 \theta} + A)^{-1} = \int_0^\infty s \exp(-e^{i\theta} st - e^{-i\theta} t A) dt$ with $\|e^{-zA}\| \leq C (\cos \theta)^{-\beta},\: \|s\exp(e^{-i\theta} s (\cdot))\|_{L^1(\R_+)} \cong \cos \theta.$
\end{rem}

As for the $\HI$ calculus, there is an extended $\Ma_1$ calculus as a counterpart of \eqref{Equ Extended HI calculus}.

\begin{defi}\label{Def Maloc calculus}
Let $A$ be a $0$-sectorial operator having a bounded $\Ma_1$ calculus and $f : (0,\infty) \to \C$ such that $\|f\ddyad_n\|_{\Ma} \leq C 2^{|n|M}$ holds for $n \in \Z$ and some $C > 0$ and $M \in \N.$
Equivalently said, there exists an $N\in\N$ such that $\rho^N f \in \Ma_1$ for some $N \in \N,$ where $\rho(\lambda) = \lambda ( 1 + \lambda)^{-2}.$
Call this class of functions $\Ma_{\loc}.$
For example, the functions $t \mapsto t^\theta$ belong to $\Ma_{\loc}$ for any $\alpha > 0$ and any $\theta \in \R.$
Note that $\rho^{-N}(A)$ is a closed and densely defined operator.
Then we define the operator $f(A)$ by
\begin{align*}
D(f(A)) & = \{ x \in X:\: (\rho^Nf)(A) \in D(\rho^{-N}(A)) \} \\
f(A) & = \rho^{-N}(A) (\rho^Nf)(A).
\end{align*}
\end{defi}

Analogously to the extended $\HI$ calculus we have the following properties.

\begin{prop}\label{Prop Soaloc calculus}
Let $A$ and $f$ be as in the above definition.
\begin{enumerate}
\item $f(A)$ is a well-defined closed operator on $X$ (independent of the choice of $N$).
\item If $x \in D(A^N) \cap R(A^N),$ then $x \in D(f(A))$ and $f(A)x = (\rho^N f)(A) \rho^{-N}(A)x.$
\item $D(A^N) \cap R(A^N)$ is a core for $f(A).$ In particular, $f(A)$ is densely defined.
\item $D(f(A)) = \{x \in X :\: \lim_{n \to \infty} (\rho^N_nf)(A)x \text{ exists in }X\}$ and for $x \in D(f(A)),$ we have $f(A)x = \lim_{n \to \infty} (\rho^N_nf)(A)x.$
Here, $\rho_n(\lambda) = \frac{n}{n+\lambda} - \frac{1}{1+n \lambda}.$
\item If $g$ is a further function in $\Ma_{\loc},$ then for $u,v \in \C,$ $D(f(A)) \cap D(g(A)) \subseteq D((uf+vg)(A))$ and $(uf + vg)(A) x = uf(A) x + vg(A) x$ for $x \in D(f(A)) \cap D(g(A)).$
\item If $f,g$ are as before and $g$ satisfies $\rho^N g \in \Ma_1,$ then $D(A^{2N}) \cap R(A^{2N}) \subset D(f(A)g(A)) = D((fg)(A)) \cap D(g(A))$ and $(fg)(A)x = f(A) [g(A)x]$ for $x \in D(f(A)g(A)).$
\item If $f$ belongs to $\Hol(\Sigma_\omega)$ for some $\omega \in (0,\pi),$ then the above definition of $f(A)$ coincides with the definition from Lemma \ref{Lem Hol}.
\item If $A$ has in addition a bounded $\Ma$ calculus, then the above definition of $f(A)$ coincides with the one from Definition \ref{Def Line calculi} provided $f$ belongs to $\Ma.$
\end{enumerate}
\end{prop}

\begin{proof}
(1) - (4) can be proved in the same way as \cite[15.8 Theorem]{KuWe} and (5)-(6) can be proved as \cite[15.9 Proposition]{KuWe}.
For (7), observe that $\rho^N f$ belongs to $\HI_0(\Sigma_\omega),$ and the definition from Lemma \ref{Lem Hol} is the same as the one from Definition \ref{Def Maloc calculus}, see \cite[15.7 Definition]{KuWe}.
For (8), let $T$ be the operator $f(A)$ from Definition \ref{Def Maloc calculus} above and $S$ the operator $f(A)$ from Definition \ref{Def Line calculi}.
Observe that the definition of $(\rho^N_n f)(A)$ from Definitions  \ref{Def Line calculi} and \ref{Def ell1Ma calculus} are the same, since $\rho^N_n f$ can be approximated by a sequence $(h_n)_n \subseteq \HI_0(\Sigma_\omega)$ simultaneously in $\Ma_1$ and in $\Ma.$
Then by the multiplicativity of the $\Ma$ calculus, for $x \in D(T),$ $T x = \lim_n (\rho^N_n f)(A)x =\lim_m \rho^N_n(A) Sx = Sx,$
since $\rho^N_n(A)x \to x$ for any $x \in X,\,N\in \N$ and $n \to \infty.$
Thus, $D(T) = X$ and $T = S.$
\end{proof}

\section{Triebel-Lizorkin type decompositions}\label{Sec 3 Spectral Decomposition}

In this section we establish the main Theorem \ref{Thm PL Decomposition} in several variants.
Our starting point is an extension of the classical Paley-Littlewood decomposition
\[ \|x\|_{L^p(\R^n)} \cong \bignorm{ \left( \sum_{n \in \Z} |\ddyad_n(-\Delta) x|^2 \right)^{\frac12} }_{L^p(\R^n)} \]
to $0$-sectorial operators with a $\Ma$-calculus.

\begin{thm}\label{Thm PL Decomposition}
Let $A$ be a $0$-sectorial operator having a $\Ma$ calculus for some $\alpha > 0.$
Let further $(\ddyad_n)_{n \in \Z}$ be a homogeneous dyadic partition of unity on $\R_+$ and $(\dyad_n)_{n \in \N_0}$ an inhomogeneous dyadic partition of unity on $\R_+.$
The norm on $X$ has the equivalent descriptions:
\begin{align}
\|x\| & \cong \E \bignorm{ \sum_{n \in \Z} \epsilon_n  \ddyad_n(A) x }_{X}
 \cong \sup\left\{ \bignorm{ \sum_{n \in \Z} a_n \ddyad_n(A) x } :\: |a_n| \leq 1 \right\} \label{Equ Prop PL Decomposition} \\
\intertext{and}
\|x\| & \cong \E \bignorm{ \sum_{n \in \N_0} \epsilon_n \dyad_n(A) x }_{X} \cong \sup \left\{ \bignorm{ \sum_{n \in \N_0} a_n \dyad_n(A) x } :\: |a_n| \leq 1 \right\}.
\label{Equ Prop PL Decomposition 2}
\end{align}
In particular the series $\sum_{n \in \Z} \ddyad_n(A)x$ converges unconditionally in $X$ for all $x \in X.$
\end{thm}

A key observation for the proof of the above theorem is the following lemma.

\begin{lem}\label{Lem Overlapping}
Let $\beta > 0$ and $(g_n)_n$ be a bounded sequence in $\Bbii$ such that for some $a > 0$ and $N \in \N,$ the supports satisfy
\[ \text{card}\{ n :\: \supp g_n \cap [x-a,x+a] \neq \emptyset \} \leq N \text{ for all }x \in \R. \]
Then $\sum_{n = 1}^\infty g_n(x)$ has finitely many terms for every $x \in \R,$ belongs to $\Bbii$ and
\[ \bignorm{ \sum_{n = 1}^\infty g_n }_{\Bbii} \lesssim_{N,a} \sup_{n \in \N} \| g_n \|_{\Bbii} .\]
\end{lem}

\begin{proof}
Recall the description \eqref{Equ Bbii norm} of the $\Bbii$ norm.
Note that $\triangle_h^M g(x)$ depends only on $g(x),g(x+1\cdot h),\ldots,g(x+Mh).$
Thus, by assumption of the lemma, for $|h| \leq \frac{a}{M},$ there exist $g_{j_1},\ldots,g_{j_N}$ such that
\[|\triangle_h^M (\sum_n g_{n})(x)| = |\triangle_h^M (\sum_{j=1}^N g_{n_j}) (x)|
= |\sum_{j=1}^N \triangle_h^M g_{n_j} (x)| \leq N \sup_n \|\triangle_h^M g_{n}\|_\infty.\]
Hence
\[\sup_{|h| \in (0,\delta)} |h|^{-\beta} \| \Delta_h^M \sum_n g_n\|_\infty \leq N \sup_{|h| \in (0,\delta)} \sup_n |h|^{-\beta} \| \Delta_h^M g_n \|_\infty.\]
Similarly, $\sup_{x \in \R}|\sum_n g_n(x)| \leq N \sup_n \|g_n\|_\infty,$ so that by \eqref{Equ Bbii norm},
\[\|\sum_n g_n\|_{\Bes^\beta_{\infty,\infty}}
\lesssim \sup_n\|g_n\|_\infty + \sup_n \sup_{|h| \in (0,\delta)} |h|^{-\beta} \|\Delta^M_h g_n\|_\infty
\cong \sup_n \|g_n\|_{\Bes^\beta_{\infty,\infty}}.\]
\end{proof}

\begin{proof}[Proof of Theorem \ref{Thm PL Decomposition}]
Choose some $\beta > \alpha.$
The idea of the proof is the following observation.
Assume that $g_n$ is a sequence as in Lemma \ref{Lem Overlapping}.
Then for any choice of signs $a_n = \pm 1,$ we have
\begin{equation}\label{Equ Paley-Littlewood idea}
\bignorm{ \sum_n a_n g_n\circ\log(A) x }_X \lesssim \sup_n |a_n| \, \|g_n\|_{\Bbii} \|x\| \lesssim \sup_n \|g_n\|_{\Bes^\beta_{\infty,\infty}} \|x\|.
\end{equation}
Let us now give the details of the proof.
Let $(a_n)_n$ be a sequence such that $|a_n| \leq 1.$
Then $g_n = a_n \ddyad_{n,e} \in C^\infty_c \subset \Bes^\beta_{\infty,\infty}$ satisfies the assumptions of Lemma \ref{Lem Overlapping}
so that $\ddyad_{(a_n)} := \sum_{n \in \Z} a_n \ddyad_{n,e} \in \Bes^\beta_{\infty,\infty}$ and
\[\|\ddyad_{(a_n)}\|_{\Bes^\beta_{\infty,\infty}} \lesssim \sup_n \| a_n \ddyad_{n,e}\|_{\Bes^\beta_{\infty,\infty}}
\leq \sup_n \|\ddyad_{n,e}\|_{\Bes^\beta_{\infty,\infty}} = \|\ddyad_{0,e}\|_{\Bes^\beta_{\infty,\infty}}.\]
By the same argument, also the partial sums $\sum_{n=-N}^M a_n \ddyad_{n,e}$ are bounded in $\Bes^\beta_{\infty,\infty},$ so that by
Proposition \ref{Prop Convergence Bes}, $\sum_{n =-N}^M a_n\ddyad_n(A)x$ converges to $\ddyad_{(a_n)}\circ\log(A)x$ as $N,M \to \infty.$
Thus,
\[ \|\sum_{n \in \Z} a_n \ddyad_n(A) x\| = \|\ddyad_{(a_n)}\circ\log(A) x\| \lesssim \|\ddyad_{(a_n)}\|_{\Bes^\beta_{\infty,\infty}} \|x\|
\lesssim \|\ddyad_{0,e}\|_{\Bes^\beta_{\infty,\infty}} \|x\|.\]
Since $|\epsilon_n(\omega)| = 1$ for any $n \in \Z$ and $\omega \in \Omega_0,$ the estimate
\[\E\|\sum_{n \in \Z} \epsilon_n \ddyad_n(A) x \|_{X} \leq \sup \left\{ \| \sum_{n \in \Z} a_n \ddyad_n(A) x \|  :\: |a_n| \leq 1 \right\}\]
is clear.
The converse inequality follows by duality, writing
$|\spr{x}{x'}| = |\sum_{n} \spr{\ddyad_n(A) x}{x'}| = |\E  \spr{\sum_{n} \epsilon_n \ddyad_n(A)x}{\sum_k \epsilon_k \widetilde\ddyad_k(A)'x'}|
\lesssim \E \| \sum_n \epsilon_n  \ddyad_n(A)x\|_{X} \E \| \sum_k \epsilon_k \widetilde\ddyad_k(A)'x'\|_{X'}$
and estimating the dual expression.
We have shown \eqref{Equ Prop PL Decomposition}.
The equivalence \eqref{Equ Prop PL Decomposition 2} can be proved similarly.
\end{proof}

\begin{rem}\label{Rem not injective}
If $A$ is not injective or $R(A)$ is not dense on a reflexive space $X$ (e.g. if $A$ is the Laplace Beltrami operator on $L^p(M),$ where $M$ is a compact Riemannian manifold), then we use the decomposition \eqref{Equ sectorial injective}, $X = \overline{R(A)} \oplus N(A)$ with $Px = \lim_{\lambda \to 0} \lambda R(\lambda,A)x$ the projection onto $N(A)$ and the part $A_1$ of $A$ on $\overline{R(A)}$ to obtain a modification of \eqref{Equ Prop PL Decomposition}
\[ \|x\| \cong \|Px\| + \E \left\| \sum_{n \in \Z} \epsilon_n \ddyad_n(A_1) x \right\| \]
and analogously for \eqref{Equ Prop PL Decomposition 2}.
If $A$ is a Hilbert space and one takes $\dyad_n(A)$ as defined by the functional calculus of self-adjoint operators then $\dyad_n(A)|_{N(A)} = P$ and \eqref{Equ Prop PL Decomposition 2} holds as it stands.
\end{rem}

The norm equivalence for $\| \Id_X(x) \|$ in \eqref{Equ Prop PL Decomposition} of Theorem \ref{Thm PL Decomposition} can be extended to possibly unbounded operators $g(A),$
if $g_2$ does not vary too much on intervals $[n-1,n+1].$
Recall that $g_2(t) = g(2^t).$
This is the content of the next proposition and will be a tool to consider fractional domain spaces of sectorial operators later on in this section.

\begin{prop}\label{Prop PL extended}
Let $A$ be a $0$-sectorial operator having a $\Ma$ calculus.
Assume that $g : \R_+ \to \C$ is invertible and such that $g_2 \in \Baloc$ and $g^{-1}_2$ also belongs to $\Baloc.$
Assume that for some $\beta > \alpha,$
\begin{equation}\label{Equ Prop PL extended}
\sup_{n \in \Z} \| (\widetilde\ddyad_n g)_2 \|_{\Bbii} \cdot \| (\ddyad_n g^{-1})_2  \|_{\Bbii} < \infty.
\end{equation}
Let $(c_n)_{n \in \Z}$ be a sequence in $\C \backslash \{ 0 \}$ satisfying $|c_n| \cong \| (\widetilde\ddyad_n g)_2 \|_{\Bbii}.$
Then for any $x \in D(g(A)),$ $\sum_{n \in \Z} c_n \ddyad_n(A) x$ converges unconditionally in $X$ and
\begin{equation}\label{Equ 2 Prop PL extended}
\| g(A)x \| \cong \E \bignorm{ \sum_{n \in \Z} \epsilon_n c_n \ddyad_n(A) x }
\cong \sup \left\{ \bignorm{ \sum_{n \in \Z} a_n c_n \ddyad_n(A) x }_X :\: |a_n| \leq 1 \right\}.
\end{equation}
\end{prop}

\begin{proof}
Let us show the unconditional convergence of $\sum_n c_n \ddyad_n(A)x$ for $x \in D(g(A)).$
By Proposition \ref{Prop Soaloc calculus} (6), we have $\ddyad_n(A)x = (g^{-1}\ddyad_n)(A) g(A) x.$
Thus for any choice of scalars $|a_n| \leq 1,$
\[\sum_{n = -N}^N a_n c_n \ddyad_n(A)x = \left[\sum_{n=-N}^N a_n c_n (g^{-1}\ddyad_n)\right](A) g(A)x.\]
The term in brackets is a sequence of functions indexed by $N$ which clearly converges pointwise for $N \to \infty.$
$[\ldots]_2$ is also uniformly bounded in $\Bes^\beta_{\infty,\infty},$ because
\begin{equation}\label{Equ 2 Proof Prop PL Baloc}
\|\sum_{-N}^N a_n c_n (g^{-1} \ddyad_{n})_2\|_{\Bes^\beta_{\infty,\infty}} \lesssim \sup_n |c_n| \, \|(g^{-1} \ddyad_{n} )_2\|_{\Bes^\beta_{\infty,\infty}}
\lesssim \sup_n \|(g \widetilde\ddyad_{n})_2\|_{\Bes^\beta_{\infty,\infty}} \, \|(g^{-1} \ddyad_{n})_2 \|_{\Bes^\beta_{\infty,\infty}} < \infty
\end{equation}
by assumption \eqref{Equ Prop PL extended}.
Thus, Proposition \ref{Prop Convergence Bes} yields the unconditional convergence.
Estimate \eqref{Equ 2 Proof Prop PL Baloc} also shows that
\[\| \sum_n a_n c_n \ddyad_n(A) x \| = \| \sum_n a_n c_n (g^{-1}\ddyad_n)(A) g(A)x \| \lesssim \sup_n |c_n| \, \|(g^{-1}\ddyad_{n})_2\|_{\Bes^\beta_{\infty,\infty}}\|g(A)x\|\]
for any choice of scalars $|a_n| \leq 1,$ so that one inequality in \eqref{Equ 2 Prop PL extended} is shown.
For the reverse inequality, we argue by duality similar to the proof of Theorem \ref{Thm PL Decomposition}.
\end{proof}

Proposition \ref{Prop PL extended} can be used to characterize the domains of fractional powers of $A.$

\begin{thm}\label{Thm Fractional power norm}
Let $A$ be a $0$-sectorial operator having a bounded $\Ma$ calculus for some $\alpha > 0.$
Let further $(\ddyad_n)_{n \in \Z}$ be a dyadic partition of unity on $\R_+$ and $(\dyad_n)_{n \in \N_0}$ be the corresponding inhomogeneous partition on $\R_+.$
Then for $\theta \in \R,$
\begin{align}
\| x \|_{\dot{X}_\theta} \cong \sup_{F \subset \Z\text{ finite}} \E \bignorm{\sum_{n \in F} \epsilon_n 2^{n \theta} \ddyad_n(A) x } \quad \text{ (homogeneous decomposition)}\label{Equ hom norm}
\intertext{and for $\theta > 0,$}
\| x \|_{X_\theta} \cong \sup_{F \subset \N_0 \text{ finite}} \E \bignorm{\sum_{n \in F} \epsilon_n 2^{n \theta} \dyad_n(A) x } \quad \text{ (inhomogeneous decomposition)}\label{Equ inhom norm}
\end{align}
\end{thm}

\begin{rem}
Recall that for $x \in D(A^\theta),$ $\|x\|_\theta = \| A^\theta x \|_X.$
These norms reduce for $A = -\Delta$ on $L^p(\R^n)$ to the classical Triebel-Lizorkin space norms. 
Therefore we consider $\dot{X}_\theta$, the completion of $D(A^\theta)$ in this norm, as generalized Triebel-Lizorkin spaces for the operator $A.$

Since our assumption implies that $A$ has bounded imaginary powers, we emphasize that the spaces $X_\theta$ form a complex interpolation scale.
If $X$ does not contain $c_0,$ then the sum on the right hand side of \eqref{Equ inhom norm} converges in $X$ (see \cite{KaW2,vN}).
\end{rem}

\begin{proof}
The operator $\ddyad_n(A) : X \to X$ can be continuously extended to $\dot{X}_\theta \to \dot{X}_\theta,$ and if $\theta \geq 0,$ to $X_\theta \to X_\theta.$
Indeed, by Proposition \ref{Prop Soaloc calculus}, for $x \in D(A^\theta),$
\[ \|\ddyad_n(A) x\|_{\dot{X}_\theta} = \|A^\theta \ddyad_n(A)x\|_X = \|\ddyad_n(A) A^\theta x\|_X \leq \|\ddyad_n(A)\|_{X\to X} \|A^\theta x\|_X,\]
whence $\| \ddyad_n(A) \|_{\dot{X}_\theta \to \dot{X}_\theta} \leq \| \ddyad_n(A) \|$ and also $\| \ddyad_n(A) \|_{X_\theta \to X_\theta} \leq \| \ddyad_n(A) \|.$
One even has $\ddyad_n(A) (\dot{X}_\theta) \subset X.$
Similarly, $\dyad_n(A)$ maps $X_\theta \to X_\theta$ for $\theta \geq 0.$
Then by the density of $D(A^\theta)$ in $\dot{X}_\theta,$ \eqref{Equ hom norm} and \eqref{Equ inhom norm} follow from
\begin{align}
\|A^\theta x\| & \cong \E \bignorm{ \sum_{n \in \Z} \epsilon_n  2^{n\theta} \ddyad_n(A) x}_{X} & (x \in D(A^\theta))
\label{Equ hom norm 2} \\
\intertext{and for $\theta > 0,$}
\|A^\theta x\| + \|x\| & \cong \E \bignorm{ \sum_{n \in \N_0} \epsilon_n 2^{n\theta} \dyad_n(A) x}_{X} & (x \in D(A^\theta)),
\label{Equ inhom norm 2}
\end{align}
which we will show now.
Let $\beta > \alpha$ and
$g(t) = 2^{t\theta}.$
Recall the embedding $C^m_b \hookrightarrow \Bes^\beta_{\infty,\infty} \hookrightarrow C^0_b$ for $m \in \N,\,m > \beta > 0$ from Proposition \ref{Prop Elementary Mih Hor}.
We have for all $n \in \Z$
\[ 2^{n\theta} \leq \|g\widetilde\ddyad_{n,2}\|_{C^0_b}
\lesssim \|g\widetilde\ddyad_{n,2}\|_{\Bes^\beta_{\infty,\infty}}
\lesssim \|g\widetilde\ddyad_{n,2}\|_{C^m_b}
\lesssim 2^{n\theta}
\]
and
$\|g^{-1} \ddyad_{n,2}\|_{\Bes^{\beta}_{\infty,\infty}} \lesssim \|g^{-1} \ddyad_{n,2}\|_{C^m_b} \lesssim 2^{-n\theta}.$
(Here, equivalence constants may depend on $\theta.$)
Consequently, $\sup_{n \in \Z} \| g \widetilde\ddyad_{n,2} \|_{\Bes^\beta_{\infty,\infty}} \| g^{-1} \ddyad_{n,2} \|_{\Bes^\beta_{\infty,\infty}} < \infty.$
By Proposition \ref{Prop PL extended}, we have with $c_n = 2^{n\theta},$
\[ \|A^\theta x \| \cong \E \Bignorm{\sum_{n \in \Z} \epsilon_n 2^{n\theta} \ddyad_n(A)x}\]
for $x \in D(A^\theta),$ so that \eqref{Equ hom norm 2} follows.\\

\noindent
By \cite[Lemma 15.22]{KuWe} (set $A =A^\theta$ and $\alpha = 1$ there),
the left hand side of \eqref{Equ inhom norm 2} satisfies
\begin{equation}\label{Equ 3 Proof Cor fractional}
\|A^\theta x\| + \|x\| \cong \|(1+A^\theta)x\| \quad (x \in D(A^\theta)),
\end{equation}
whereas the right hand side of \eqref{Equ inhom norm 2} is equivalent to
\begin{equation}\label{Equ 2 Proof Cor fractional}
\E \Bignorm{\sum_{n \geq 1} \epsilon_n 2^{n \theta} \dyad_n(A)x} + \|\dyad_0(A)x\|.
\end{equation}

\noindent
``$\lesssim$'' in \eqref{Equ inhom norm 2}:
We use the equivalent expressions from \eqref{Equ 3 Proof Cor fractional} and \eqref{Equ 2 Proof Cor fractional}.
We set $g(t) = 1 + 2^{t\theta}.$
Then one checks similarly to the first part that $\|g\widetilde\ddyad_{n,2}\|_{\Bes^\beta_{\infty,\infty}} \cong \max(1,2^{n\theta})$ for $n \in \Z$ and
that $\sup_{n \in \Z} \|g\widetilde\ddyad_{n,2}\|_{\Bes^\beta_{\infty,\infty}} \|g^{-1}\ddyad_{n,2}\|_{\Bes^\beta_{\infty,\infty}} < \infty.$
Thus, by Proposition \ref{Prop PL extended},
\begin{align}
\|(1+A^\theta)x\| &  \cong \E \Bignorm{ \sum_{n \in \Z} \epsilon_n \max(1,2^{n\theta})\ddyad_n(A)x}\nonumber\\
& \leq \E \Bignorm{ \sum_{n \leq -1} \ldots} + \E \Bignorm{ \sum_{n \geq 1} \ldots } + \|\ddyad_0(A)x\|.\label{Equ Proof Cor fractional}
\end{align}
We estimate the three summands.
We have $\ddyad_n \dyad_0 = \ddyad_n$ for any $n \leq -1,$ so that
\[ \E \bignorm{ \sum_{n \leq -1} \epsilon_n \ddyad_n(A)x } = \E \bignorm{ \sum_{n \leq -1} \epsilon_n \ddyad_n(A) \dyad_0(A)x }
\lesssim \|\dyad_0(A)x\|,
\]
where we use \eqref{Equ Prop PL Decomposition} in the last step.
Since $\ddyad_n = \dyad_n$ for $n \geq 1,$ also the second summand is controlled by \eqref{Equ 2 Proof Cor fractional}.
Finally, we have $\ddyad_0 = \ddyad_0 [ \dyad_0 + \ddyad_1],$ so that
\[\|\ddyad_0(A) x \| \leq \|\dyad_0(A)x\| + \|\ddyad_0(A) \ddyad_1(A) x\| \lesssim \|\dyad_0(A)x\| + \|\ddyad_1(A)x\|,\]
and the last term is controlled by the second summand of \eqref{Equ Proof Cor fractional}.
This shows ``$\lesssim$'' in \eqref{Equ inhom norm 2}.\\

\noindent
``$\gtrsim$'' in \eqref{Equ inhom norm 2}:
We use again the expression in \eqref{Equ 2 Proof Cor fractional}.
By the first part of the theorem,
\[\E \Bignorm{ \sum_{n \geq 1} \epsilon_n 2^{n\theta} \dyad_n(A)x } \leq \E \Bignorm{ \sum_{n \in \Z} \epsilon_n 2^{n\theta} \ddyad_n(A)x} \lesssim \|A^\theta x\|.\]
Finally, $\|\dyad_0(A)x\| \lesssim \|x\|$ because $\dyad_0$ belongs to $\Ma.$
\end{proof}

The next goal is a continuous variant of Theorem \ref{Thm Fractional power norm}.
We have two preparatory lemmas.

\begin{lem}\label{Lem Mihlin calculus and Differentiation}
Let $\beta > \alpha > 0$ and $A$ be a $0$-sectorial operator with bounded $\Ma$ calculus.
\begin{enumerate}
\item
If $g :\: (0,\infty) \to \C$ is such that $g_e \in \Bbii,$ then for any $x\in X,$
$t \mapsto g(tA)x$ is continuous.

\item
If $g :\: (0,\infty) \to \C$ is such that $g_e \in \Bes^{\beta+1}_{\infty,\infty},$ then for any $x \in X$ and $t > 0$
\[ \frac{d}{dt} \left[g(tA)x\right] = A g'(tA)x. \]
\end{enumerate}
\end{lem}

\begin{proof}
(1) Since $\beta > 0,$ $g$ is a continuous function and thus, for any $t \in \R,$ $g_e(t+h) \to g_e(t)$ as $h \to 0.$
Also $\sup_{h \neq 0} \|g_e(\cdot + h)\|_{\Bes^\beta_{\infty,\infty}} = \|g_e\|_{\Bes^\beta_{\infty,\infty}} < \infty,$
so that we can appeal to Proposition \ref{Prop Convergence Bes} with $f_n = g(\cdot + h_n)$ and $h_n$ a null sequence.\\

\noindent
(2)
Fix some $x \in X$ and $t_0 \in \R.$
As $\Bes^{\beta + 1}_{\infty,\infty} \hookrightarrow C^1_b$ by Proposition \ref{Prop Elementary Mih Hor}, $g_e$ is continuously differentiable.
Hence for any $t \in \R,$ $\lim_{h \to 0} \frac{1}{h}(g_e(t_0 + t + h) - g_e(t_0 + t)) = (g_e)'(t_0 + t).$
Further, $\frac{1}{h} (g_e(t_0 + \cdot + h) - g_e(t_0 + \cdot))$ is uniformly bounded in $\Bes^\beta_{\infty,\infty}$ \cite[Section 2.3, Proposition]{RuSi}.
Then the claim follows at once from Proposition \ref{Prop Convergence Bes}.
\end{proof}

The $\HI$ calculus variant of the following lemma is a well-known result of McIntosh (see \cite{McIn}, \cite[Lemma 9.13]{KuWe}, \cite[Theorem 5.2.6]{Haasa}).

\begin{lem}\label{Lem McIntosh Integral}
Let $A$ be a $0$-sectorial operator having a bounded $\Ma$ calculus and $D_A \subset X$ be its calculus core from \eqref{Equ D}.
Let further $g : (0,\infty) \to \C$ be a function with compact support (not containing $0$) such that $g_e \in \Bes^\beta_{\infty,\infty}$ for some $\beta > \alpha.$
Assume that $\int_0^\infty g(t) \frac{dt}{t} = 1.$
Then for any $x \in D_A,$
\begin{equation}\label{Equ Lem McIntosh Integral}
x = \int_0^\infty g(tA)x \frac{dt}{t}.
\end{equation}
\end{lem}

\begin{proof}
Let $x \in D_A.$
Then there exists $\rho \in C^\infty_c(\R_+)$ such that $\rho(A)x = x.$
As $g$ has by assumption compact support, there exist $b > a > 0$ such that
\begin{equation}\label{Equ 2 Proof McIntosh Integral}
g(tA)\rho(A) = 0 \quad (t \in [a,b]^c).
\end{equation}
By Lemma \ref{Lem Mihlin calculus and Differentiation}, $t \mapsto g(tA)x$ is continuous, and \eqref{Equ 2 Proof McIntosh Integral} implies that
$\int_0^\infty g(tA) x \frac{dt}{t} = \int_a^b g(tA)x \frac{dt}{t}.$
Also,
\[\left(\int_a^b g(t\cdot) \frac{dt}{t}\right)(A)x = \left( \int_a^b g(t\cdot) \frac{dt}{t} \rho\right)(A)x =\left(\int_0^\infty g(t\cdot) \frac{dt}t \rho\right)(A) x = \rho(A) x = x.\]
In the third equality we have used the assumption $\int_0^\infty g(t) \frac{dt}{t} = 1,$ which extends by substitution to $\int_0^\infty g(ts) \frac{dt}{t} = 1\quad (s > 0).$
It remains to show
\[ \left( \int_a^b g(t\cdot) \frac{dt}{t} \right)(A)x = \int_a^b g(tA)x \frac{dt}{t}. \]
If $g$ belongs to $\HI(\Sigma_\omega)$ for some $\omega,$ then this is shown in \cite[Lemma 9.12]{KuWe}.
For a general $g,$ we appeal to the approximation from Lemma \ref{Lem HI dense in diverse spaces}.
\end{proof}

For the next theorem we recall that a Banach space has cotype $q$ if there is a constant $C$ so that for all $x_1,\ldots,x_n \in X$
\[ \left( \sum_k \| x_k \|^q \right)^{\frac1q} \leq C \E \left\| \sum_k \epsilon_k x_k \right\|. \]
All closed subspaces of a space $L^p(U,\mu)$ have cotype $q = \max(2,p).$

\begin{thm}\label{Thm Fractional power continuous}
Let $A$ be a $0$-sectorial operator having a bounded $\Ma$ calculus on $X.$
Suppose that the space $X$ and its dual $X'$ both have finite cotype.
Let $\theta \in \R$ and $\psi : (0,\infty) \to \C$ be a non-zero function such that for some $C,\: \epsilon > 0$ and $M > \alpha + 1,$ we have
\begin{align}
\sup_{k=0,\ldots,M} |t^{k-\theta} \psi^{(k)}(t)| & \leq C \min(t^{\epsilon},t^{-\epsilon}) \quad (t > 0) \label{Equ Thm Fractional power continuous} \\
\intertext{and}
\int_0^\infty |t^{-\theta} \psi(t)|^2 \frac{dt}{t} & < \infty. \label{Equ Thm Fractional power continuous 2}
\end{align}
Then we have
\begin{align}
 \| x \|_{\dot{X}_\theta} & \cong \|t^{-\theta} \psi(tA) x \|_{\gamma(\R_+,\frac{dt}{t},X)}
= \left( \E \left\| \sum_{k=1}^\infty \gamma_k \int_{\R_+} h_k(t) t^{-\theta} \psi(tA)x \frac{dt}{t} \right\|_X^2 \right)^{\frac12}. \label{Equ hom frac cont} \\
\intertext{Here $(\gamma_k)_{k \in \N}$ is an i.i.d. sequence of standard Gaussian random variables and $(h_k)_{k \in \N}$ is an orthonormal basis of $L^2(\R_+,\frac{dt}{t}).$
If $\theta > 0,$ $\psi$ satisfies \eqref{Equ Thm Fractional power continuous} and also} \sup_{k=0,\ldots,M} |t^{k} \psi^{(k)}(t)| & \leq C \min(t^{\epsilon},t^{-\epsilon}) \quad (t > 0),\nonumber\\
\intertext{then}
\|x\|_{X_\theta} & \cong \|(t^{-\theta} + 1) \psi(tA)x\|_{\gamma(\R_+,\frac{dt}{t},X)}. \label{Equ inhom frac cont}
\end{align}
\end{thm}

If $A$ is not injective, then we refer to Remark \ref{Rem not injective} how to modify \eqref{Equ hom frac cont} and \eqref{Equ inhom frac cont} in this case.

\begin{rem}
We remark that \eqref{Equ Thm Fractional power continuous} and \eqref{Equ Thm Fractional power continuous 2} are satisfied e.g. for
\begin{align*}
\psi_{\exp}(t) & = t^a \exp(-t^b) \quad (a,\:b > 0,\: \frac{a}{b} > \theta)
\intertext{and}
\psi_{\text{res}}(t) & = t^a (\lambda - t)^{-b} \quad (\lambda \in \C \backslash [0,\infty),\: b > 0 \text{ and }\theta < a < b + \theta).
\end{align*}
\end{rem}

\begin{proof}
By density of $D(A^\theta)$ in $\dot{X}_\theta$ and in $X_\theta,$ it suffices to show \eqref{Equ hom frac cont} and \eqref{Equ inhom frac cont} for $x \in D(A^\theta).$
We first reduce \eqref{Equ hom frac cont} to the case $\theta = 0.$
Set temporarily $\psi_\theta(t) = t^{-\theta} \psi(t).$
Then $\psi$ satisfies the hypotheses of the proposition for $\theta$ if and only if $\psi_\theta$ does for $0.$
By Proposition \ref{Prop Soaloc calculus} (6), for $x \in D(A^\theta),$ $t^{-\theta}\psi(tA)x = \psi_\theta(tA) A^\theta x.$
Thus, if \eqref{Equ hom frac cont} holds for $\theta = 0,$
also
\[\|t^{-\theta}\psi(tA)x\|_{\gamma(\R_+,\frac{dt}{t},X)} = \|\psi_\theta(tA)A^\theta x\|_{\gamma(\R_+,\frac{dt}{t},X)} \cong \|A^\theta x\| \quad (x \in D(A^\theta)).\]
Assume now $\theta = 0.$
By Lemma \ref{Lem Mihlin calculus and Differentiation} and the fundamental theorem of calculus, we have for $n \in \Z$ and $t \in [2^n,2^{n+1})$
\[ \psi(tA) x = \psi(2^n A)x + \int_1^2 \chi_{[2^n,t]}(2^n s) 2^ns A \psi'(2^nsA) x \frac{ds}{s}.\]
Writing $\chi_n = \chi_{[2^n,2^{n+1})}$ and $\psi(tA) x = \sum_{n \in \Z} \chi_n(t) \psi(tA)x,$ this yields by Lemma \ref{Lem Technical gamma spaces} (2) (that the assumption there is satisfied follows easily from \eqref{Equ 2 Proof Prop Fractional power continuous} below):
\[
\|\psi(tA) x\|_{\gamma(\R_+,\frac{dt}{t},X)}  \leq \Bignorm{ \sum_{n \in \Z} \chi_n(t) \psi(2^n A)  x }_{\gamma(\R_+,\frac{dt}{t},X)} + \int_1^2 \Bignorm{ \sum_{n \in \Z} \chi_n(t)  2^nsA \psi'(2^nsA) x}_{\gamma(\R_+,\frac{dt}t,X)} \frac{ds}{s}.
\]
Since the $\chi_n$ are orthonormal in $L^2(\R_+,\frac{dt}{t})$ and $\|\chi_n\|_{L^2(\R_+,\frac{dt}{t})}$ does not depend on $n,$ we have
\[ \bignorm{ \sum_{n \in \Z}\chi_n(t) \psi(2^n A)  x }_{\gamma(\R_+,\frac{dt}{t},X)}
 \cong \E \bignorm{ \sum_{n \in \Z} \gamma_n \psi(2^nA)x} \cong \E \bignorm{ \sum_{n \in \Z} \epsilon_n  \psi(2^nA)x},
\]
where the last equivalence follows from the fact that $X$ has finite cotype.
The last expression can be estimated by $\|x\|$ according to \eqref{Equ Classical Mihlin condition} and \eqref{Equ Paley-Littlewood idea},
 provided that for some $C > 0$ and for any choice of scalars $a_n = \pm 1,$ we have
$\bignorm{ \sum_{n \in \Z} a_n \psi(2^n \cdot) }_{\Ma} \leq C.$
We have for $M > \alpha$ that
\begin{align*}
\bignorm{ \sum_{n \in \Z} a_n \psi(2^n \cdot) }_{\Ma} & \lesssim \sup_{t > 0} \sup_{k=0,\ldots,M} t^k \left| \sum_{n \in \Z} a_n 2^{nk} \psi^{(k)}(2^n t) \right| \\
& \leq \sup_{t \in [1,2]} \sup_{k = 0,\ldots,M} t^k \sum_{n \in \Z} \left| 2^{nk} \psi^{(k)}(2^n t) \right| \\
& \leq C \sum_{n \in \Z} 2^{-|n| \epsilon} 2^\epsilon \\
& < \infty.
\end{align*}
Replacing $\psi$ by $\psi_1 = s (\cdot)\psi'(s(\cdot)),$ by the same arguments, we also have with $M > \alpha + 1$ that
\begin{equation}\label{Equ 2 Proof Prop Fractional power continuous}
\bignorm{ \sum_{n \in \Z} \chi_n(t)  2^nsA \psi'(2^nsA) x}_{\gamma(\R_+,\frac{dt}t,X)} \leq C \|x\|.
\end{equation}
Note that $\|\psi_1\|_{\Ma}$ is independent of $s,$ and thus also the above constant $C$ is.
We have shown that $\|\psi(tA)x\|_{\gamma(\R_+,\frac{dt}{t},X)} \leq c_1 \|x\|.$\\

For the reverse inequality, we assume first that $x$ belongs to the calculus core $D_A.$
By Lemma \ref{Lem McIntosh Integral},
\[ c x = \int_0^\infty |\psi|^2(tA) x \frac{dt}{t}\]
with $c = \displaystyle \int_0^\infty |\psi(t)|^2 \frac{dt}{t} \in (0,\infty).$
Thus, by Lemma \ref{Lem Technical gamma spaces}, for any $x' \in X',$
\[ |\spr{x}{x'}| = | c^{-1}\int_0^\infty \spr{\psi(tA)x}{\overline{\psi}(tA)'x'} \frac{dt}{t} |
 \lesssim \|\psi(tA)x\|_{\gamma(\R_+,\frac{dt}{t},X)} \|\psi(tA)'x'\|_{\gamma(\R_+,\frac{dt}{t},X')}.
\]
Now proceed as in the first part, noting that $X'$ has finite cotype, and deduce $\|\psi(tA)'x'\| \lesssim \|x'\|.$
This shows $\|x\| \leq c_2 \|\psi(tA)x\|_\gamma$ for $x \in D_A.$
For a general $x \in X,$ let $(x_n)_n \subset D_A$ with $x_n \to x.$
Then
\[\|\psi(tA)x\| \geq \|\psi(tA)x_n\| - \|\psi(tA)(x-x_n)\| \geq c_2^{-1} \|x_n\| - c_1 \|x-x_n\|,\]
and letting $n \to \infty$ shows \eqref{Equ hom frac cont} for all $x \in X.$
Finally, \eqref{Equ inhom frac cont} is a simple consequence of \eqref{Equ hom frac cont}.
Just note that the right hand side of \eqref{Equ inhom frac cont} satisfies
\[\|(1+t^{-\theta}) \psi(tA) x\| \cong \|\psi(tA)x\| + \|t^{-\theta} \psi(tA)x\|.\]
Indeed, ``$\lesssim$'' is the triangle inequality and ``$\gtrsim$'' follows from the two inequalities
\[\|\psi(tA)x\|_\gamma \leq \|(1+t^{-\theta})^{-1}\|_\infty \|(1+t^{-\theta})\psi(tA)x\|_\gamma\]
and
\[\|t^{-\theta}\psi(tA)x\|_\gamma \leq \|t^{-\theta}(1+t^{-\theta})^{-1}\|_\infty \|(1+t^{-\theta})\psi(tA)x\|_\gamma.\]
\end{proof}

In the above theorems, the Paley-Littlewood decomposition was deduced from a $\Ma$ calculus of $A.$
The following proposition gives a result in the converse direction.
For a sufficient criterion for the $R$-bounded $\Ma_1$ calculus in the proposition see Proposition \ref{Prop Resolvent Ma0}.

\begin{prop}\label{Prop weak to strong calculus}
Let $A$ be a $0$-sectorial operator with an $R$-bounded $\Ma_1$ calculus, that is $\{ f(A):\: f \in \Ma_1,\: \|f\|_{\Ma_1} \leq 1 \}$ is $R$-bounded, and let $A$ have a Paley-Littlewood decomposition, i.e. \eqref{Equ Prop PL Decomposition} holds, in particular, $\sum_{n \in \Z} \ddyad_n(A) x$ converges unconditionally in $X.$
Then $A$ has a bounded $\Ma$ calculus.
\end{prop}

\begin{proof}
Let $f \in \HI_0(\Sigma_\omega)$ for some $\omega \in (0,\pi).$
Let further $(\ddyad_n)_n$ be a dyadic partition of unity on $\R_+.$
Then $f\ddyad_n$ belongs to $\Ma_1$ for any $n \in \Z$ and $\sup_{n \in \Z} \|f\ddyad_n\|_{\Ma} \lesssim \|f\|_{\Ma}.$
Using the assumptions we deduce for $x \in X$
\begin{align*} \|f(A) x\| & \cong \E \bignorm{ \sum_{n \in \Z} \epsilon_n (f \ddyad_n)(A) \widetilde{\ddyad_n}(A) x } \\
& \leq R((f \ddyad_n)(A) :\: n \in \Z) \E \bignorm{ \sum_{n \in \Z} \epsilon_n \widetilde{\ddyad_n}(A) x} \\
& \lesssim \|f\|_{\Ma} \| x \|.
\end{align*}
Thus the proposition follows from Lemma \ref{Lem HI dense in diverse spaces} along with the well-known approximation of arbitrary $\HI(\Sigma_\omega)$ functions 
by functions as $f$ above.
\end{proof}

A Banach space has by definition type $2$ if there is a constant $C$ so that for all $x_1,\ldots,x_n \in X$
\[ \E \left\| \sum_k \epsilon_k x_k \right\| \leq C \left( \sum_k \| x_k \|^2 \right)^{\frac12}.\]
All closed subspaces of a space $L^p(U,\mu)$ with $p \geq 2$ have type $2.$
As an immediate consequence of Theorems \ref{Thm PL Decomposition} and \ref{Thm Fractional power continuous}, we obtain an extension of the results of \cite{IP} to our general setting.

\begin{cor}
Let $X$ be a Banach space of type 2 and $A$ a $0$-sectorial operator with a bounded $\Ma$ calculus for some $\alpha > 0.$
If $(\ddyad_n),(\dyad_n)$ and $\theta,\psi$ are as in Theorems \ref{Thm PL Decomposition} and \ref{Thm Fractional power continuous} respectively, then we have the inequalities

\begin{align*}
\|x\| & \leq C \left( \sum_{n \in \Z} \| \ddyad_n(A) x \|^2 \right)^{\frac12} \\
\|x\| & \leq C \left( \sum_{n=0}^\infty \| \dyad_n(A) x \|^2 \right)^{\frac12} \\
\|x\|_\theta & \leq C \left( \int_0^\infty \| t^{-\theta} \psi(tA) x\|^2 \frac{dt}{t} \right)^{\frac12}.
\end{align*}

\end{cor}

If $X = L^p(U,\mu),\:p \geq 2,$ then by \eqref{Equ PL equivalence} these inequalities follow from Theorems \ref{Thm PL Decomposition} and \ref{Thm Fractional power continuous} simply by Minkowski's inequality.
For general $X$ we refer to \cite{KaW2}.

\section{Besov type decomposition}\label{Sec 5 Real Interpolation}

We now turn to the description of real interpolation spaces in the scales $\dot{X}_\theta$ and $X_\theta.$
For $A = -\Delta$ on $L^p(\R^d),$ these interpolation spaces correspond to homogeneous and inhomogeneous Besov spaces.
Abstract Besov spaces have been described in \cite[Theorem 3.6.2]{AmBG}
in the case that $A$ generates a $C_0$-group with polynomial growth,
and in \cite{HaaInterpolation} in the case that $A$ is a sectorial operator with a bounded $\HI$ calculus.
See also \cite{Zh} for certain operators on $L^p$ self-adjoint on $L^2,$ with kernel estimates of $\dyad_n(A).$

Throughout the rest of the section we assume that $A$ is a $0$-sectorial operator with bounded $\Ma_1$ calculus.
Let $(\ddyad_n)_n$ and $(\dyad_n)_n$ be a homogeneous and an inhomogeneous partition of unity on $\R_+.$
We introduce the spaces $\dB^\theta_q(A)$ and $\B^\theta_q(A)\: (\theta \in \R,\:q \in [1,\infty])$ to be
\begin{align*} \dB^\theta_q(A) & = \left\{ x \in \dot{X}_{-N} + \dot{X}_N :\: \|x\|_{\dB^\theta_q(A)} = \left( \sum_{n \in \Z} 2^{n \theta q} \|\ddyad_n(A)x\|^q \right)^{\frac1q} < \infty \right\}
 \intertext{and}
\B^\theta_q(A) & = \left\{ x \in \dot{X}_{-N} + \dot{X}_N :\: \|x\|_{\B^\theta_q(A)} = \left( \sum_{n \in \N_0} 2^{n \theta q} \|\dyad_n(A)x\|^q \right)^{\frac1q} < \infty \right\},
\end{align*}
(standard modification if $q = \infty$),
where $-N < \theta < N.$
Often we write $\dB^\theta_q$ for $\dB^\theta_q(A)$ and $\B^\theta_q$ for $\B^\theta_q(A)$ if confusion seems unlikely.
Note that as remarked at the beginning of the proof of Theorem \ref{Thm Fractional power norm}, $\ddyad_n(A)$ is a bounded operator on $\dot{X}_N$ provided $A$ has an $\Ma_1$ calculus.
One even has that $\ddyad_n(A)$ maps $X_N$ into $X,$ so that $x_m = \sum_{k = -m}^m \ddyad_k(A)x$ belongs to $D_A$ for $x \in \dB^\theta_q.$
It is easy to check that $x_m$ converges to $x$ in $\dB^\theta_q$ for $q < \infty,$ so that $D_A$ is a dense subset in $\dB^\theta_q$ and the definition of $\dB^\theta_q$ is independent of $N.$
Similarly, $\widetilde{D_A} = \{ x \in X :\: \exists \:N \in \N:\: \dyad_n(A)x = 0 \text{ for }n \geq N\}$ is dense in $\B^\theta_q$ for $q < \infty$ and the definition of $\B^\theta_q$ is independent of $N.$
Indeed, it only remains to check that $\dyad_0(A)$ belongs to $B(X).$

To this end, write $\dyad_0(t) = \exp(-t) + (\dyad_0(t) - \exp(-t)) = \exp(-t) + \sum_{n \in \Z} (\dyad_0(t) - \exp(-t)) \ddyad_n(t).$
Let $g_n(t) = (\dyad_0(t) - \exp(-t)) \ddyad_n(t).$
We estimate the $\Ma$ norm of $g_n$ separately for $n \leq -2,\,-1 \leq n \leq 1$ and $2 \leq n.$
If $n \leq -2,$ then $g_n(t) = (1 - \exp(-t)) \ddyad_n(t)$ and an elementary calculation shows that $|g_n(t)| \leq C 2^{n}$ and also $| t^k g_n^{(k)}(t) | \leq C 2^{n}.$
Then by \eqref{Equ Classical Mihlin condition}, $\|g_n\|_{\Ma} \leq C 2^{n}.$
If $-1 \leq n \leq 1,$ then $\|g_n\|_{\Ma} \leq C,$ whereas if $2 \leq n,$ then
$g_n(t) =  - \exp(-t) \ddyad_n(t),$ so that $|t^k g_n^{(k)}(t)| \leq C 2^{k(n+1)} \exp(-2^{n-1}).$
It follows that $\|g_n\|_{\Ma} \leq C 2^{(\lfloor \alpha \rfloor + 1)(n+1)} \exp(-2^{n-1}).$
Summing up, we obtain
\[ \sum_{n \in \Z} \|g_n\|_{\Ma} \leq \sum_{n \leq -2} C2^{n} + \sum_{n=-1}^n C + \sum_{n \geq 2} C 2^{(\lfloor \alpha \rfloor + 1)(n+1)} \exp(-2^{n-1}) < \infty. \]
Thus by the $\Ma_1$ calculus, $\|\dyad_0(A)\| \leq \|\exp(-A)\| + \sum_{n \in \Z} \|g_n(A)\| < \infty.$
Now it is easy to check that also $\dyad_0(A) : X_N \to X.$

If $A$ is not injective, then we refer to Remark \ref{Rem not injective} how to modify $\|x\|_{\dB^\theta_q}$ and $\|x\|_{\B^\theta_q}$ in this case.

\begin{prop}\label{Prop Besov interpolation scale}
If $A$ has a bounded $\Ma_1$ calculus, then the spaces $\dB^\theta_q\:(\theta \in \R,\: q \in [1,\infty])$ form a real interpolation scale, i.e.
$(\dB^{\theta_0}_{q_0},\dB^{\theta_1}_{q_1})_{\vartheta,q} = \dB^{\theta_\vartheta}_q$ with $\vartheta \in (0,1),$ $\theta_\vartheta = \theta_1 \vartheta + \theta_0 ( 1 - \vartheta )$ and $\theta_0 \neq \theta_1.$
Also the spaces $\B^\theta_q\: (\theta \in \R,\: q \in [1,\infty])$
form a real interpolation scale, i.e.
$(\B^{\theta_0}_{q_0},\B^{\theta_1}_{q_1})_{\vartheta,q} = \B^{\theta_\vartheta}_q.$ 
\end{prop}

\begin{proof}
We start with the homogeneous Besov type spaces.
We proceed in a similar manner to the classical Besov spaces, see e.g. \cite[2.4.2]{Triea}.
Note that $\dB^{\theta_0}_{q_0},\dB^{\theta_1}_{q_1} \hookrightarrow \dot{X}_N + \dot{X}_{-N},$ where $N > \max(|\theta_0|,|\theta_1|),$
so that taking real interpolation spaces is meaningful.

In a first step, we show that $(\dB^{\theta_0}_\infty,\dB^{\theta_1}_\infty)_{\vartheta,q} \hookrightarrow \dB^{\theta_\vartheta}_q.$
Let $(\ddyad_n)_{n \in \Z}$ be a homogeneous dyadic partition of unity on $\R_+.$
Let $x = x_0 + x_1$ with $x_0 \in \dB^{\theta_0}_\infty$ and $x_1 \in \dB^{\theta_1}_\infty.$
Then 
\begin{align*}
2^{k\theta_0} \| \ddyad_k(A)x\|_X & \leq 2^{k\theta_0} \|\ddyad_k(A)x_0\|_X + 2^{k(\theta_0 - \theta_1)} 2^{k\theta_1} \|\ddyad_k(A)x_1\|_X \\
& \leq \|x_0\|_{\dB^{\theta_0}_\infty} + 2^{k(\theta_0-\theta_1)} \|x_1\|_{\dB^{\theta_1}_\infty}.
\end{align*}
Thus, if $q < \infty,$
\begin{align*}
\sum_{k \in \Z} 2^{q k \theta_\vartheta} \|\ddyad_k(A)x\|^q & = \sum_{k \in \Z} 2^{q k \vartheta (\theta_1 - \theta_0)} 2^{qk\theta_0} \|\ddyad_k(A)x\|^q \\
& \leq \sum_{k\in\Z} 2^{-q k \vartheta ( \theta_0 - \theta_1 )} K^q(2^{k(\theta_0-\theta_1)},x;\dB^{\theta_0}_\infty,\dB^{\theta_1}_\infty) \\
& \leq C_{\theta_0 - \theta_1} \int_0^\infty t^{-q \vartheta} K^q(t,x;\dB^{\theta_0}_\infty,\dB^{\theta_1}_\infty) \frac{dt}{t} \\
& \cong \|x\|^q_{(\dB^{\theta_0}_\infty,\dB^{\theta_1}_\infty)_{\vartheta,q}},
\end{align*}
where $K(t,x;\dB^{\theta_0}_\infty,\dB^{\theta_1}_\infty)$ stands for the usual $K$-functional of the real interpolation method, and the estimate of the sum against the integral follows from the fact that $K(t,x;\dB^{\theta_0}_\infty,\dB^{\theta_1}_\infty)$ is positive, increasing and concave, see \cite[3.1.3. Lemma]{BeL}.
If $q = \infty,$ then the above calculation must be modified in an obvious way.

In a second step, we show that if $1 \leq r \leq q,$ then $\dB^{\theta_\vartheta}_q \hookrightarrow (\dB^{\theta_0}_r,\dB^{\theta_1}_r)_{\vartheta,q}.$
We can assume without loss of generality that $\theta_0 > \theta_1.$
Then also $\theta_\vartheta > \theta_1.$
We have for $x \in \dB^{\theta_\vartheta}_q$
$\|x\|_{(\dB^{\theta_0}_r,\dB^{\theta_1}_r)_{\vartheta,q}} \cong \sum_{k \in \Z} 2^{-\vartheta q k (\theta_0 - \theta_1)} K^q(2^{k(\theta_0-\theta_1)},x;\dot{B}^{\theta_0}_r,\dB^{\theta_1}_r).$
For $k \in \Z$ fixed, we choose the decomposition $x = x_0 + x_1$ with $x_0 = (\sum_{j= - \infty}^k \ddyad_j)(A)x$ and $x_1 = (\sum_{j=k+1}^\infty \ddyad_j)(A)x.$
Note that $(\sum_{j= - \infty}^k \ddyad_j)(A)$ and $(\sum_{j=k+1}^\infty \ddyad_j)(A)$ are bounded on $\dot{X}_m$ for any $m \in \R.$
Indeed, right before the statement of the proposition, we have shown that $\dyad_0(A)$ is bounded on $X.$
Thus, e.g. for $k > 0,$ also $(\sum_{j= - \infty}^k \ddyad_j)(A) = \dyad_0(A) + \sum_{j=1}^k \ddyad_j(A)$ is bounded on $X,$
and $(\sum_{j=k+1}^\infty \ddyad_j)(A) = \Id_X - (\sum_{j= - \infty}^k \ddyad_j)(A)$ is bounded on $X.$
Then it is easy to check that they are also bounded on $\dot{X}_m.$
Note that by Proposition \ref{Prop Soaloc calculus}, $\ddyad_l(A) (\sum_{j=-\infty}^k \ddyad_j(A)) = \ddyad_l(A) \sum_{j \leq k:|j-l|\leq 1} \ddyad_j(A).$
We have
\[ \|x_0\|_{\dB^{\theta_0}_r}^r = \sum_{l \in \Z} 2^{l \theta_0 r} \| \ddyad_l(A) \sum_{j=-\infty}^k \ddyad_j(A) x \|^r
\lesssim \sum_{l=-\infty}^{k+1} 2^{l \theta_0 r} \| \ddyad_l(A) x \|^r \]
and
\[ \|x_1\|_{\dB^{\theta_1}_r}^r = \sum_{l \in \Z} 2^{l \theta_1 r} \| \ddyad_l(A) \sum_{j=k+1}^\infty \ddyad_j(A) x \|^r
\lesssim \sum_{l=k}^\infty 2^{l \theta_1 r} \| \ddyad_l(A) x \|^r. \]
Thus, if $q < \infty,$
\[ \sum_{k \in \Z} 2^{-\vartheta q k (\theta_0 - \theta_1)} K^q(2^{k(\theta_0 - \theta_1)},x;\dB^{\theta_0}_r,\dB^{\theta_1}_r)
\lesssim \sum_{k \in \Z} 2^{q k \theta_\vartheta} \left[ \sum_{l=-\infty}^{k+1} 2^{(l-k)\theta_0 r} \|\ddyad_l(A)x\|^r
+ \sum_{l=k}^\infty 2^{(l-k)\theta_1 r} \|\ddyad_l(A)x\|^r \right]^{\frac{q}{r}}. \]
Choose now $\theta_1 < \kappa_1 < \theta_\vartheta < \kappa_0 < \theta_0.$
Apply two times H\"older's inequality with $\frac{r}{q} + \frac{r}{\sigma} = 1$  (modification below if $r = q,\,\sigma = \infty$).
This gives, interchanging summation over $k$ and $l$ after the second estimate,
\begin{align*}
&\sum_{k \in \Z} 2^{-\vartheta q k (\theta_0 - \theta_1)} K^q(2^{k(\theta_0 - \theta_1)},x;\dB^{\theta_0}_r,\dB^{\theta_1}_r) \lesssim \sum_{k \in \Z} 2^{q k (\theta_\vartheta - \theta_0)} \left[ \sum_{l=-\infty}^{k+1} 2^{l\sigma(\theta_0 - \kappa_0)} \right]^{\frac{q}{\sigma}} \left[ \sum_{l=-\infty}^{k+1} 2^{l\kappa_0 q} \| \ddyad_l(A)x\|^q \right] \\
& + \sum_{k\in\Z} 2^{q k (\theta_\vartheta - \theta_1)} \left[ \sum_{l=k}^\infty 2^{l \sigma (\theta_1 - \kappa_1)} \right]^{\frac{q}{\sigma}}
\left[ \sum_{l=k}^\infty 2^{l \kappa_1 q} \| \ddyad_l(A)x\|^q \right] \\
& \lesssim \sum_{l \in \Z} 2^{l \kappa_0 q} \|\ddyad_l(A)x\|^q \sum_{k=l-1}^\infty 2^{q k(\theta_\vartheta - \theta_0)} 2^{kq(\theta_0 - \kappa_0)} + \sum_{l \in \Z} 2^{l \kappa_1 q} \|\ddyad_l(A)x\|^q \sum_{k=-\infty}^l 2^{q k(\theta_\vartheta - \theta_1)} 2^{kq(\theta_1 - \kappa_0)} \\
& \cong \sum_{l\in\Z} 2^{l \kappa_0 q} 2^{lq (\theta_\vartheta - \kappa_0)} \|\ddyad_l(A)x\|^q +  \sum_{l\in\Z} 2^{l \kappa_1 q} 2^{lq (\theta_\vartheta - \kappa_1)} \|\ddyad_l(A)x\|^q \\
& \cong \|x\|^q_{\dB^{\theta_\vartheta}_q}.
\end{align*}
If $q = \infty,$ then the above calculation has to be adapted in a straightforward way.
Now combining both steps, we have, since $\dB^{\theta_k}_{q_0} \hookrightarrow \dB^{\theta_k}_{q_1}$ for $k=0,1$ and $q_0 \leq q_1,$ and $(Y,Z)_{\vartheta,q} \hookrightarrow (Y_1,Z_1)_{\vartheta,q}$ for $Y \hookrightarrow Y_1$ and $Z \hookrightarrow Z_1,$
\[ \dB^{\theta_\vartheta}_q \hookrightarrow (\dB^{\theta_0}_r,\dB^{\theta_1}_r)_{\vartheta,q} \hookrightarrow (\dB^{\theta_0}_{q_0},\dB^{\theta_1}_{q_1})_{\vartheta,q} \hookrightarrow (\dB^{\theta_0}_\infty,\dB^{\theta_1}_\infty)_{\vartheta,q} \hookrightarrow  \dB^{\theta_\vartheta}_q. \]
This shows the statement on the homogeneous Besov spaces.
The proof for the spaces $\B^\theta_q$ is similar, see also \cite[2.4.2]{Triea}.
\end{proof}

\begin{thm}\label{Thm Besov identification}
Let $A$ be a $0$-sectorial operator with $\Ma_1$ calculus.
Then we have the following identifications of the real interpolation of fractional domain spaces:
\begin{align*}
(\dot{X}_{\theta_0},\dot{X}_{\theta_1})_{\vartheta,q} & = \dB^{\theta_\vartheta}_q \quad (\theta_0 \neq \theta_1,\: \theta_\vartheta = \theta_1 \vartheta + \theta_0 (1-\vartheta)) \\
\intertext{and}
({X}_{\theta_0},{X}_{\theta_1})_{\vartheta,q} & = \B^{\theta_\vartheta}_q \quad (\theta_0,\theta_1 \geq 0,\: \theta_0 \neq \theta_1,\: \theta_\vartheta = \theta_1 \vartheta + \theta_0 (1-\vartheta)).
\end{align*}
\end{thm}

\begin{proof}
We show that $\dB^\theta_1 \subset \dot{X}_\theta \subset \dB^\theta_{\infty}.$
Then by Proposition \ref{Prop Besov interpolation scale} and the Reiteration theorem \cite[Theorem 3.5.3]{BeL} it follows that 
$(\dot{X}_{\theta_0},\dot{X}_{\theta_1})_{\vartheta,q} = (\dB^{\theta_0}_{q_0},\dB^{\theta_1}_{q_1})_{\vartheta,q} = \dB^{\theta_\vartheta}_q$
for $\theta_0 \neq \theta_1,$ $\theta_\vartheta = \theta_1 \vartheta + \theta_0 (1-\vartheta)$ and $q \in [1,\infty].$
For $x$ belonging to the dense subspace $D_A$ of $\dB^\theta_1,$ we have
\[ \| A^\theta x \| \leq \sum_{n \in \Z} \|A^\theta \ddyad_n(A) x\| = \sum_{n \in \Z} \| A^\theta \widetilde{\ddyad_n}(A) \ddyad_n(A) x \|
 \leq \sup_{k \in \Z} 2^{-k\theta} \|\lambda^\theta \widetilde{\ddyad_k}(\lambda)\|_{\Ma} \sum_{n \in \Z} 2^{n \theta} \|\ddyad_n(A) x\|,
\]
and $2^{-k\theta} \|\lambda^\theta \widetilde{\ddyad_k}(\lambda)\|_{\Ma} = 2^{-k\theta} \| \lambda^{\theta} \widetilde{\ddyad_0}(2^{-k}\lambda)\|_{\Ma} 
= 2^{-k\theta} \| (2^k \lambda)^\theta \widetilde{\ddyad_0}(\lambda) \|_{\Ma} = \|\lambda^\theta \widetilde{\ddyad_0}(\lambda) \|_{\Ma} < \infty.$
This shows that $\dB^{\theta}_1 \subset \dot{X}_\theta.$
Conversely, let $x \in D(A^\theta).$
Then 
\[2^{n \theta} \|\ddyad_n(A) x\| = 2^{n \theta} \|A^{-\theta} \ddyad_n(A) A^\theta x\| \leq \sup_{n \in \Z} 2^{n \theta} \| \lambda^{-\theta} \ddyad_n(\lambda) \|_{\Ma} \| A^\theta x \|, \]
and, similarly to the above, $2^{n \theta} \| \lambda^{-\theta} \ddyad_n(\lambda) \|_{\Ma} = \|\lambda^{-\theta} \ddyad_0(\lambda)\|_{\Ma} < \infty.$
This shows that $\dot{X}_\theta \subset \dB^\theta_\infty.$
The proof for the inhomogeneous spaces is similar, using the same estimate for $\|\dyad_0(A)\|$ as before Proposition \ref{Prop Besov interpolation scale}.
\end{proof}

The analogous statement of the above theorem for a continuous parameter reads as follows.

\begin{thm}\label{Thm Besov identification continuous}
Let $A$ have an $\Ma_1$ calculus, $\theta_0,\: \theta_1 \in \R$ with $\theta_0 < \theta_1,\: s \in (0,1),$ $q \in [1,\infty],$ and $\theta = (1-s) \theta_0 + s \theta_1.$
Furthermore, let $f : (0,\infty) \to \C$ be a function with $\sum_{k \in \Z}\|f \ddyad_0(2^{-k}\cdot)\|_{\Ma_1} 2^{-k\theta} < \infty$ and $f^{-1} \ddyad_0 \in \Ma_1.$
Then, with standard modification for $q = \infty,$ the following hold.
\begin{enumerate}
\item For the real interpolation space $\dB^\theta_q = (\dot{X}_{\theta_0}, \dot{X}_{\theta_1} )_{s,q},$ we have the norm equivalence
\[ \|x\|_{\dB^\theta_q} \cong \| x \|_{s,q} \cong \left( \int_0^\infty t^{-\theta q} \| \ddyad_0(tA) x \|^q \frac{dt}{t} \right)^{\frac1q} \cong \left(\int_0^\infty t^{-\theta q} \|f(tA)x\|^q \frac{dt}{t} \right)^{\frac1q}.\]
\item For the real interpolation space $\B^\theta_q,$ we have the norm equivalence
\[ \| x \|_{\B^\theta_q} \cong \left( \int_0^1 t^{-\theta q} \| \ddyad_0(tA) x \|^q \frac{dt}{t} \right)^{\frac1q} + \| \dyad_0(A)x \|.\]
\end{enumerate}
\end{thm}

\begin{rem}\label{Rem Besov continuous}
\begin{enumerate}~
\item If $A$ is not injective then we refer to Remark \ref{Rem not injective} how to modify $\|x\|_{s,q}$ and $\|x\|_{\B^\theta_q}$ in 1. and 2. of the Theorem in this case.
\item If $\theta > 0,$ we can replace in 2. $\|\dyad_0(A)x\|$ by $\|x\|$ and/or replace the integration from $\int_0^1$ to $\int_0^a$ with $a \in [1,\infty].$
If we choose $a = \infty,$ we can also replace in the integrand $\|\ddyad_0(tA)x\|$ by $\|f(tA)x\|,$ where $f$ satisfies the assumptions of Theorem \ref{Thm Besov identification continuous}.
\item Common choices for $f$ are functions in $\HI_0(\Sigma_\sigma)$, e.g. $f(\lambda) = \lambda^a e^{-\lambda}$ for $a > \theta,$
or $f(\lambda) = \lambda^a (1 + \lambda)^{-b},$ with $0 < a - \theta < b < \infty.$
In the latter cases, one obtains square functions of the form
\begin{align*} & \left( \int_0^\infty t^{-\theta q} \|(tA)^a e^{-tA} x\|^q \frac{dt}{t} \right)^{\frac1q} \\
\intertext{or}
& \left(\int_0^\infty t^{-\theta q} \|t^a A^a (1+tA)^{-b}x\|^q \frac{dt}{t} \right)^{\frac1q}.
\end{align*}
\end{enumerate}
\end{rem}

\begin{proof}[Proof of Theorem \ref{Thm Besov identification continuous}]
(1)
To simplify notations, we assume that $\theta_0 = 0, \: \theta_1 = 1.$
Let $n \in \Z$ such that $t \in [2^n, 2^{n+1} [.$
Then $\| \ddyad_0(tA) x \| = \| \sum_{k = n - 1}^{n+2} \ddyad_0(2^k A) \ddyad_0(tA) x \|.$
As $\| \ddyad_0(tA) x \| \leq C \| \ddyad_0 \|_{\Ma} \| x \|,$ one finds $\|\ddyad_0(tA) x \| \leq C \sum_{|k-n| \leq 2} \| \ddyad_0(2^k A) x \|,$ and thus
\begin{align*}
\left( \int_0^\infty t^{-\theta q} \| \ddyad_0(tA) x \|^q \frac{dt}{t} \right)^{\frac1q} & \cong \left( \sum_{n \in \Z} \int_{2^n}^{2^{n+1}} 2^{-n \theta q} \| \ddyad_0(tA) x \|^q \frac{dt}{t} \right)^{\frac1q} \\
& \lesssim \left( \sum_{n \in \Z} 2^{-n \theta q} \sum_{|k-n| \leq 2} \| \ddyad_0(2^k A)x \|^q \int_{2^n}^{2^{n+1}} 1 \frac{dt}{t} \right)^{\frac1q} \\
& \leq \left( 5 \sum_{n \in \Z} 2^{-n \theta q } \| \ddyad_0(2^n A) x \|^q \right)^{\frac1q} \cdot 2^{2 |\theta|}.
\end{align*}
On the other hand, one has
$ \| \ddyad_0(2^n A) x \| \lesssim \| \sum_{k = - 2}^2 \ddyad_0(t 2^k A) \ddyad_0(2^n A)x \| $ for any $t \in [2^n, 2^{n+1}].$
Thus, $2^{-n \theta q} \| \ddyad_0(2^n A) x \|^q \lesssim \int_{2^n}^{2^{n+1}} t^{-\theta q} \sum_{k = - 2}^2 \| \ddyad_0 (t2^k A)x \|^q \frac{dt}{t},$ so that
\begin{align*}
\left( \sum_{n \in \Z} 2^{-n \theta q} \| \ddyad_0(2^n A)x\|^q \right)^{\frac1q} & \lesssim \left( \int_0^\infty t^{-\theta q} \sum_{k=-2}^2 \| \ddyad_0(t2^k A) x\|^q \frac{dt}{t} \right)^{\frac1q} \\
& \cong \left( \int_0^\infty t^{-\theta q} \| \ddyad_0(tA) x \|^q \frac{dt}{t} \right)^{\frac1q}.
\end{align*}
Similarly, one obtains the second norm equivalence in 1. for the case $q = \infty.$\\

We now show that for $x \in \dB^\theta_q,$ we have
\begin{equation}
\label{equ proof Besov continuous}
\int_0^\infty \left( t^{-\theta} \|\ddyad_0(tA) x\| \right)^q \frac{dt}{t} \cong \int_0^\infty \left( t^{-\theta} \|f(tA) x\| \right)^q \frac{dt}{t}
\end{equation}
under the assumptions on $f$ from the Theorem.
For the inequality ``$\lesssim$'', we use that $\lambda \mapsto \frac{1}{f(\lambda)} \ddyad_0(\lambda)$ has finite $\Ma_1$ norm.
Thus,
\begin{align*}
\int_0^\infty & \left( t^{-\theta} \|\ddyad_0(tA) x\| \right)^q \frac{dt}{t} = \int_0^\infty \left( t^{-\theta} \|(f^{-1}(\lambda) \ddyad_0(\lambda))|_{\lambda = tA} f(tA)x\| \right)^q \frac{dt}{t} \\
& \lesssim \int_0^\infty \left( t^{-\theta} \|f(tA)x\| \right)^q \frac{dt}{t}.
\end{align*}
For the reverse inequality ``$\gtrsim$'', we estimate
\begin{align*}
& \left\{ \int_0^\infty \left( t^{-\theta} \|f(tA)x\|\right)^q \frac{dt}{t} \right\}^{\frac{1}{q}} \leq \sum_{k \in \Z} \left\{ \int_0^\infty \left( t^{-\theta} \|\ddyad_k(tA) f(tA)x\| \right)^q \frac{dt}{t} \right\}^{\frac{1}{q}} \\
& = \sum_{k \in \Z} \left\{ \int_0^\infty \left( t^{-\theta} \| \widetilde{\ddyad_k}(\lambda) f(\lambda)|_{\lambda = tA} \ddyad_k(tA) f \| \right)^q \frac{dt}{t} \right\}^{\frac{1}{q}} \\
& \leq \sum_{k \in \Z} \|f(\lambda) \widetilde{\ddyad_0}(2^{-k} \lambda)\|_{\Ma_1} \left\{  \int_0^\infty \left( t^{-\theta} \|\ddyad_k(tA) f \| \right)^q \frac{dt}{t} \right\}^{\frac{1}{q}} \\
& = \sum_{k \in \Z} \|f(\lambda) \widetilde{\ddyad_0}(2^{-k} \lambda)\|_{\Ma_1} \left\{ \int_0^\infty \left( t^{-\theta} 2^{-k\theta} \|\ddyad_0(tA) f \| \right)^q \frac{dt}{t} \right\}^{\frac{1}{q}} \\
& = \sum_{k \in \Z} \|f(\lambda) \widetilde{\ddyad_0}(2^{-k} \lambda)\|_{\Ma_1} 2^{-k \theta} \left\{ \int_0^\infty \left( t^{-\theta} \|\ddyad_0(tA) f \| \right)^q \frac{dt}{t} \right\}^{\frac{1}{q}}.
\end{align*}
Using the assumptions on $f,$ we have shown \eqref{equ proof Besov continuous}.

(2)
Assume $q < \infty.$
Repeating the arguments of part (1), we get
\begin{align*}
\|x\|_{\B^\theta_q} & \lesssim \| \dyad_0(A) x \| + \left( \sum_{n = - \infty}^0 2^{-n \theta q} \| \ddyad_0(2^n A) x \|^q \right)^{\frac1q} \\
& \lesssim \| \dyad_0(A)x \| + \left( \int_0^1 t^{-\theta q} \| \ddyad_0(tA) x \|^q \frac{dt}{t} \right)^{\frac1q} \\
& \lesssim \|x\|_{\B^\theta_q} + \left( \sum_{n=0}^\infty 2^{n\theta q} \|\ddyad_n(A)x\|^q \right)^{\frac1q} \\
& \lesssim \|x\|_{\B^\theta_q}.
\end{align*}
The case $q = \infty$ is treated similarly.
\end{proof}

\begin{proof}[Proof of Remark \ref{Rem Besov continuous}]
2. 
Assume $\theta > 0.$
We have $\|\dyad_0(A)x\| \lesssim \|x\|$ and $\int_0^a t^{-\theta q} \|\ddyad_0(tA)x\|^q \frac{dt}{t} \leq \int_0^\infty t^{-\theta q} \|\ddyad_0(tA)x\|^q \frac{dt}{t}.$
On the other hand, since $\B^\theta_q \hookrightarrow X_{\theta_0} + X_{\theta_1} \hookrightarrow X,$ we have $\|x\| \lesssim \|x\|_{\B^\theta_q}.$
Also, 
\begin{align*}
\int_a^\infty t^{-\theta q} & \|\ddyad_0(tA)x\|^q \frac{dt}{t} = \int_a^\infty t^{-\theta q} \| \ddyad_0(tA) \sum_{k=0}^{N_a}\dyad_k(A)x\|^q \frac{dt}{t} \\
& \lesssim \| \ddyad_0\|_{\Ma_1} \int_a^\infty t^{-\theta q} \frac{dt}{t} \sum_{k=0}^{N_a}\|\dyad_k(A)x\| \\
& \lesssim \|x\|_{\B^\theta_q}^q,
\end{align*}
where we have used that $\theta > 0.$
Finally, the proof of Theorem \ref{Thm Besov identification continuous} part 1. shows that one can replace $\|\ddyad_0(tA)x\|$ by $\|f(tA)x\|$ in the integral, in the case $a = \infty.$

3. It is an easy matter to check with \eqref{Equ Classical Mihlin condition} that $\|\lambda^a e^{-\lambda} \widetilde{\ddyad_0}(2^{-k} \lambda\|_{\Ma_1} \lesssim \min(2^{ka},\exp(-c 2^k)),$ so that $\sum_{k \in \Z} \|\lambda^a e^{-\lambda} \widetilde{\ddyad_0}(2^{-k} \lambda\|_{\Ma_1} 2^{-k \theta} < \infty$ for $\theta < a.$
Likewise, the assumptions in Theorem \ref{Thm Besov identification continuous} are checked for $f(\lambda) = \lambda^a (1 + \lambda)^{-b}$ for $0 < a - \theta < b.$
\end{proof}

Let $A$ be a $0$-sectorial operator having a bounded $\Ma_1$ calculus.
Let further $\theta \in \R$ and $q \in [1,\infty)$ and consider the real interpolation space $\dB^\theta_q.$
Let $R_\lambda :\: D_A \subset \dB^\theta_q \to \dB^\theta_q,\, x \mapsto (\lambda - A)^{-1} x$ which extends to a bounded operator with $\|R_\lambda\| \leq \|(\lambda - A)^{-1}\|_{X \to X}.$
It is easy to check that $\C \backslash (-\infty,0] \to B(\dB^\theta_q),\: \lambda \mapsto R_\lambda$ is a pseudo-resolvent.
As $R(R_\lambda) \supset D_A,$ $R_\lambda$ has dense range and consequently, is the resolvent of a closed and densely defined operator \cite[Theorem 9.3]{Pazy} which we denote by $\dot{A}.$
Furthermore, $\dot{A}$ is $0$-sectorial in $\dB^\theta_q$ and by \eqref{Equ Cauchy Integral Formula} and density arguments, $f(\dot{A})x = f(A)x$ for any $x \in D_A$ and $f \in \Ma.$
Similarly, for $\theta \geq 0,$ there is a $0$-sectorial operator $\widetilde{A}$ on $\B^\theta_q$ such that $f(\widetilde{A})x = f(A)x$ for any $x \in D_A$ and $f \in \Ma.$

\begin{thm}
Let $A$ have an $\Ma_1$ calculus.
Then $\dot{A}$ has a $\Ma$ calculus on $\dB^\theta_q$ for any $\theta \in \R$ and $q \in [1,\infty).$
\end{thm}

\begin{proof}
Let $f \in \Ma$ and $x \in D_A.$
Then
\begin{align*}
\| f(\dot{A})x \|_{\dB^\theta_q} & = \left( \sum_{n \in \Z} 2^{n \theta q} \|f(A) \ddyad_n(A) x\|^q \right)^{\frac1q} \\
& = \left( \sum_{n \in \Z} 2^{n \theta q} \|(f\ddyad_n)(A) \widetilde{\ddyad_n}(A) x\|^q \right)^{\frac1q} \\
& \leq \sup_{k \in \Z} \|(f\ddyad_k)(A)\|_{X \to X} \left( \sum_{n \in \Z} 2^{n \theta q} \|\widetilde{\ddyad_n}(A) x\|^q \right)^{\frac1q} \\
& \lesssim \sup_{k \in \Z} \|f \ddyad_k\|_{\Ma} \| x \|_{\dB^\theta_q} \\
& \lesssim \|f\|_{\Ma} \| x \|_{\dB^\theta_q}.
\end{align*}
Thus the theorem follows.
\end{proof}

\begin{rem}
Let $A$ have an $\Ma_1$ calculus, $\theta \in \R$ and $q \in [1,\infty).$
The inhomogeneous Besov space $\B^\theta_q(A)$ of $A$ coincides with the homogeneous Besov space $\dB^\theta_q(A+1)$ of $A+1.$
Then $(A+1)\dot{\phantom{i}}$ has a $\Ma$ calculus on $\dB^\theta_q(A+1)$ by the above theorem.
Since $f(A+1) = f_1(A)$ for $f_1(t) = f(t+1),$ we have thus $\|f(\widetilde{A}) \|_{\B^\theta_q(A) \to \B^\theta_q(A)} \leq C \|f\|_{\Ma}$ for any $f\in\Ma$ with $\supp f \subset (1,\infty).$
Furthermore, to consider functional calculus for functions with full support in $(0,\infty),$ we make the following observation.
With a similar proof of the boundedness of $\dyad_0(A)$ before Proposition \ref{Prop Besov interpolation scale}, one can check that for a function $f \in C^k(-\epsilon,1)$ for some $k > \alpha$ and $\epsilon > 0,$ we have that $f(A) = f(0)e^{-A} + (f(\cdot)-f(0)e^{-\cdot})(A)$ is a bounded operator on $X$ and $f(\widetilde{A})$ is a bounded operator on $\B^\theta_q(A).$
\end{rem}

\section{Some Examples}\label{Sec 6 Examples}

The theory developed in Sections \ref{Sec 3 Spectral Decomposition} and \ref{Sec 5 Real Interpolation} gives a unified approach to various classes of operators.
Firstly, the following lemma gives sufficient conditions for the $\Ma$ calculus of some operator $A.$
We quote some convenient sources for the required $\Ma$ calculus without striving for the best possible parameter $\alpha.$

\begin{lem}\label{Lem Criteria Examples Ma calculus}
Let $(\Omega,\mu)$ be a $\sigma$-finite measure space and $A$ a self-adjoint positive operator on $L^2(\Omega).$
\begin{enumerate}
\item\label{it Lem Criteria Examples Ma calculus}Suppose that $\Omega$ is a homogeneous space of dimension $d \in \N$ and that the $C_0$-semigroup generated by $-A$ has an integral kernel satisfying the 
upper Gaussian estimate \cite[Assumptions 2.1 and 2.2]{DuOS}.
Then, on $X = L^p(\Omega)$ for $1 < p < \infty,$ the injective part $A_1$ of $A$ from the decomposition in Subsection \ref{Subsec A B} is $0$-sectorial and has a $\Ma$ calculus for $\alpha > \lfloor \frac{d}{2} \rfloor + 1.$
\item Suppose that on $X = L^p(\Omega)$ for $1 < p < \infty,$ $A$ is injective, has a bounded $\HI$ calculus and the imaginary powers satisfy
$\|A^{it}\|_{p \to p} \lesssim ( 1 + |t| )^\beta$ for $t \in \R$ and some $\beta \geq 0.$
Then $A$ is $0$-sectorial and has a $\Ma$ calculus on $X$ for $\alpha > \beta + 1.$
\item Suppose that $\Omega$ is a homogeneous space of dimension $d \in \N$ and that the $C_0$-semigroup generated by $-A$ satisfies generalized Gaussian estimates \cite[(GGE)]{Bluna}.
Then $A$ has a $\Ma$ calculus for $\alpha > \frac{d}{2} + \frac12,$ on $L^p(\Omega)$ for $p \in (p_0,p_0')$ for some $p_0 \in [1,2).$
This has been refined in \cite{Kr,Uhl,COSY}.
\end{enumerate}
Note that in the literature of spectral multipliers, $\Ma$ calculus is very often defined by considering the self-adjoint functional calculus and then extrapolating to $L^p.$
This calculus coincides with Definition \ref{Def Line calculi} (restricted to $\overline{R(A)}$ in \eqref{it Lem Criteria Examples Ma calculus}).
\end{lem}

\begin{proof}
To a proof of (1), we refer to \cite[Theorem 7.23, (7.69)]{Ouha} and \eqref{Equ Classical Mihlin condition},
whereas (2) can be found in \cite[Theorem 6.1 (a)]{KrW} or \cite[Theorem 4.73 (a)]{Kr}.
The fact that the extrapolated self-adjoint calculus coincides with the $\Ma$ calculus is shown in \cite[Illustration 4.87]{Kr}.
\end{proof}

\subsection{Lie Groups}\label{Subsec Examples Lie groups}

Let $G$ be a connected Lie group of polynomial volume growth with respect to its Haar measure.
That is, if $U$ is a compact neighborhood of the identity element of $G,$ then there are constants $c,d > 0$
such that $V(U^k) \leq c k^d$ for any $k \in \N.$
Let further $X_1,\ldots,X_n$ be left-invariant vector fields on $G$ satisfying H\"ormander's condition, i.e. they generate together with their successive Lie brackets, at any point of $G$ the tangent space.
We put $A = - \sum_{j = 1}^n X_j^2$ which is the associated sub-laplacian.
Then $A$ has a $\Ma$ calculus for 
\[\alpha > \max (\lfloor\frac{d}{2}\rfloor, \lfloor \frac{D}{2} \rfloor ) + 1,\]
according to \cite[Theorem]{Alex}, \eqref{Equ Classical Mihlin condition} and the last part of Lemma \ref{Lem Criteria Examples Ma calculus}.
Here, $d,D \in \N$ are given by the volume growth of balls in $G,$ $V(B(x,r)) \cong r^d \quad (r \to 0)$ and $V(B(x,R)) \cong R^D \quad (R \to \infty).$
We obtain immediately the following spectral decomposition as a corollary.

\begin{cor}\label{Cor FuMV}\cite[Theorem 4.4]{FuMV}
\begin{align}
\|x\|_p & \cong \bignorm{ \left( \sum_{n \in \Z} |\ddyad_n(A) x|^2 \right)^{\frac12} }_p
\cong \bignorm{ \left( \sum_{n = 0}^\infty |\dyad_n(A) x|^2 \right)^{\frac12} }_p \label{Equ Cor FuMV} \\
& \cong \bignorm{ \left( \int_0^\infty | \psi(tA)x |^2 \frac{dt}{t} \right)^{\frac12} }_p \label{Equ 2 Cor FuMV}
\end{align}
for any $1 < p < \infty,$ where $(\ddyad_n)_{n \in \Z}$ and $(\dyad_n)_{n\in\N_0}$ are a homogeneous and an inhomogeneous dyadic partition of unity on $\R_+$,
and $\psi$ satisfies for some $C,\epsilon > 0,\,\alpha$ as above and any $t > 0$
\begin{equation}\label{Equ Cor 52}\sup_{k= 0,\ldots,\alpha + 1} |t^k \psi^{(k)} (t)| \leq C \min (t^\epsilon,t^{-\epsilon}) \quad \text{and} \quad \int_0^\infty | \psi (t)|^2 \frac{dt}{t} < \infty.\end{equation}
\end{cor}

\begin{proof}
As we remarked above, $A$ has a $\Ma$ calculus for $\alpha > \max (\lfloor\frac{d}{2}\rfloor, \lfloor \frac{D}{2} \rfloor ) + 1.$
Then the first two equivalences follow from Theorem \ref{Thm PL Decomposition} and \eqref{Equ PL equivalence}.
The third equivalence then follows from Theorem \ref{Thm Fractional power continuous} and the square function equivalence on $L^p$ from Subsection \ref{Subsec Prelims Rad}.
\end{proof}

Under certain additional conditions we can write the expression in \eqref{Equ 2 Cor FuMV} as a convolution.
Namely, we suppose that $G$ is nilpotent and possesses representations $\pi_\lambda,\:\lambda \in \R$ along with dilations $\delta_s,\: s > 0$ satisfying the following properties.
\begin{enumerate}
\item[(G1)] There exists a Hilbert space $H$ such that $\pi_\lambda : G \to B(H)$ is a unitary representation for any $\lambda \in \R.$
\item[(G2)] The induced mappings $L^1(G) \to B(H),\, f \mapsto \hat{f}(\lambda) = \int_G f(g) \pi_\lambda(g) dg$ satisfy $\|\hat{f}(\lambda)\| \leq \|f\|_1$ for any $\lambda \in \R.$
\item[(G3)] We have $\hat{f_s}(\lambda) = \hat{f}(\delta_s \lambda)$ for any $\lambda \in \R,$ where for $s > 0,\:f_s(g) = s^{-n-1} f(\delta_{s^{-\frac12}}(g))$ with suitable $n \in \N$
($2n+2$ is the homogeneous dimension of $G$).
\item[(G4)] For $\psi \in \Ma$ and $\phi \in L^1(G)$ such that $\psi(A)f = f \ast \phi = \int_G f(u) \phi(u^{-1}(\cdot))du$ is a convolution operator,
we have $\sup_{\lambda \in \R} \|\hat{\phi}(\lambda)\| = \|\psi\|_\infty.$
\item[(G5)] If $\psi,\:\phi$ are as above with $\psi(A)f = f \ast \phi,$ then for $s > 0,$ we have $\psi(sA) f = f \ast \phi_s(A).$
\end{enumerate}
This is satisfied e.g. if $G$ is the Heisenberg group $H^n = \R^n \times \R^n \times \R,$ with group law $(x,y,t)\cdot (x',y',t') = (x+x',y+y',t+t'+\frac12 [\sum_{k=1}^n x_k'y_k - x_ky_k']).$
Then, explicitely, we have the Schr\"odinger representation on $H = L^2(\R^n)$
\[ \pi_\lambda :\: \begin{cases} H^n & \longrightarrow B(L^2(\R^n)) \\
                    h(s) & \mapsto e^{i(\lambda t + \text{sign}(\lambda)\lambda^{\frac12} x \cdot s + \frac{\lambda}{2} x \cdot y)} h(s + |\lambda|^{\frac12})
                   \end{cases}
\]
for $\lambda \in \R,$ the dilations $\delta_s(x,y,t) = (sx,sy,s^2 t),\, s > 0,$ and
$A = - \sum_{k = 1}^n X_k^2 + Y_k^2 = - \sum_{k = 1}^n \left( \frac{\partial}{\partial x_k} + \frac12 y_k \frac{\partial}{\partial t} \right)^2 +
\left( \frac{\partial}{\partial y_k} - \frac12 x_k \frac{\partial}{\partial t} \right)^2$ is the sub-laplacian operator on $H^n.$
For example if $\psi$ and $\phi$ are as in (G4), then $\hat{\phi}(\lambda) = \sum_{\alpha \in \N^n} \psi( (2|\alpha| + n) |\lambda|) P_\alpha,$ where the $P_\alpha$ are projections on $H = L^2(\R^n)$
mutually orthogonal, coming from an orthonormal basis of Hermite functions on $\R^n$ \cite{LiuMa}.

We obtain the following corollary from Theorem \ref{Thm Fractional power continuous}.

\begin{cor}\label{Cor Liu Ma}
Let $G$ be a Lie group as above satisfying (G1) -- (G5).
Then for $1 < p < \infty,$ the following equivalence of $g$-function type holds
\[ \| f \|_{L^p(G)} \cong \bignorm{ \left(\int_0^\infty | s^{-n} f \ast \phi( \delta_s (\cdot) ) |^2 \frac{ds}{s} \right)^{\frac12} }_{L^p(G)}, \]
provided that $\int_0^\infty |\psi(t)|^2 \frac{dt}{t} < \infty$ and
\begin{equation}\label{Equ Kernel Assumption}
A^l \left[\frac{d^k}{ds^k} \phi(\delta_s \cdot) \right]|_{s = 1} \in L^1(G)
\end{equation}
for $l = \pm 1$ and $k = 0,\ldots, M + 1,$ where $\psi$ and $\phi$ are given by (G1) -- (G5).
\end{cor}

\begin{proof}
We show that $\psi$ satisfies the assumptions of Theorem \ref{Thm Fractional power continuous} with $\theta = 0.$
We claim that
\begin{equation}
A\psi'(A)f = f \ast [\frac{d}{ds}\phi(\delta_s (\cdot))].
\end{equation}
First note that $\frac{d}{ds} \phi(\delta_s (\cdot) )|_{s = 1} = \lim_{h \to 1} \frac{1}{h-1} \left[ \phi(\delta_h (\cdot)) - \phi \right]$ holds true with limit in $L^1(G).$
Then with $\psi_1(x) = x \psi'(x),$ we have for any $f \in L^p(G)$
\[\psi_1(A)f = \lim_{h \to 1} \frac{1}{h-1} \left[ \psi(hA) - \psi(A) \right]f = \lim_{h \to 1} \frac{1}{h-1} f \ast \left[ \phi(\delta_h(\cdot)) - \phi \right] ,\]
the first equality according to Lemma \ref{Lem Mihlin calculus and Differentiation}, the second according to (G3).
By properties (G2) and (G4), then $\sup_{t > 0}|t^{1+l} \psi'(t)| \leq \| A^l \frac{d}{ds}\phi(\delta_s \cdot) \|_1 < \infty$ for $l = \pm 1.$
This shows the first condition of \eqref{Equ Cor 52} for $\alpha = 0.$
The proof for higher orders of $\alpha$ is similar.
\end{proof}

\begin{rem}
If $G = H^n,$ then for example for $k = 1,$ \eqref{Equ Kernel Assumption} can be stated as $A^l [ \frac{x}{2} \partial_x \phi + \frac{y}{2} \partial_y \phi + t \partial_t \phi ] \in L^1(H^n).$
Note that in \cite{LiuMa}, even more general conditions on the kernel $A^{-Q/4} \phi \in L^2(H^n)$ and $|\nabla \phi(u) | \leq C ( 1 + |u| )^{-Q -1 - \epsilon},$
where $\nabla$ denotes the gradient on the Heisenberg group, are obtained to conclude Corollary \ref{Cor Liu Ma}.
See also \cite{MSt} for more results on spectral multipliers on the Heisenberg group.
\end{rem}

\subsection{Sublaplacians on Riemannian manifolds}

Let us discuss sublaplacians on Riemannian manifolds as an application of Sections \ref{Sec 3 Spectral Decomposition} and \ref{Sec 5 Real Interpolation}.
Let $M$ be a Riemannian manifold and $\Delta$ be the Laplace operator on $M.$
First of all, if $M$ is compact or asymptotically conical, then Theorem \ref{Thm PL Decomposition} holds \cite[p.~2]{Bouc}.
Note however that Theorem \ref{Thm PL Decomposition} may not hold on asymptotically hyperbolic manifolds where the volume of geodesic balls grows exponentially with respect to their radii \cite[p.~3]{Bouc}.

\subsection{Schr\"odinger Operators}\label{Subsec Schrodinger}

Let us discuss an application of Theorem \ref{Thm Fractional power norm}.
In \cite{OZ,Zh}, abstract Triebel-Lizorkin spaces associated with Schr\"odinger operators are considered.
For example the Schr\"odinger operator with P\"oschl-Teller potential $A = - \Delta + V_n,$ where 
$V_n(x) = -n(n+1) \text{sech}^2 x$ \cite{OZ}, or with magnetic potential
$A = - \sum_{j=1}^n (\frac{\partial}{\partial x_j} + i a_j)^2 + V,$
where $n \geq 3,$ $a_j(x) \in L^2_{\loc}(\R^n)$ is real valued, $V = V_+ - V_-$ with $V_+ \in L^1_{\loc}(\R^n)$ and $\|V_-\|_{K_n} < \gamma_n = \pi^{n/2}/\Gamma(\frac{n}{2} - 1),$
$\| \cdot \|_{K_n}$ standing for the Kato class norm \cite{Zh}.

If $A$ acts on $L^p(\R^d)$ and $1 < p < \infty,$ then by \eqref{Equ PL equivalence},
$\|x\|_{\dot{X}_\theta} \cong \|x\|_{\dot{F}^\theta_{p,2}(A)}$ and $\|x\|_{X_\theta} \cong \|x\|_{F^\theta_{p,2}(A)}$ with the norms on the right hand side defined in \cite[Section 1]{Zh},
which correspond to our norms \eqref{Equ hom norm} and \eqref{Equ inhom norm} in this case.
Thus we can apply Theorem \ref{Thm Fractional power norm} and obtain as a corollary
that the self-adjoint operators $A$ on $L^2(\R^d)$ considered in \cite[Section 2.4]{Zh},
including Schr\"odinger operators have fractional domain spaces 
which are independent of the choice of the partition of unity and form a complex interpolation scale.
Indeed, these operators have a $\Ma$ calculus as they satisfy the
usual upper Gaussian bound for the semigroup \cite[(5)]{Zh}, so Lemma \ref{Lem Criteria Examples Ma calculus} (1) applies.

Next we consider real interpolation spaces in the situation of \cite[Theorem 1.3]{Zh}.
If $A$ is a Schr\"odinger operator on $L^p(\R^d)$ considered in \cite[Section 2.4]{Zh}, then our spaces $\dB^\theta_q(A)$ and $\B^\theta_q(A)$ coincide
with the spaces $\dot{B}^{\theta,q}_{p}(A)$ and $B^{\theta,q}_{p}(A)$ from \cite[Section 1]{Zh} for $1 < p < \infty,\, 1 \leq q < \infty.$
This follows from the fact that $A$ has a $\Ma$ calculus by \cite[(5)]{Zh}.

We consider also Schr\"odinger operators $A = - \Delta + V$ on $L^p(\R^d)$ where the potential $V = V_+ - V_-$ is such that the positive part $V_+$ belongs to $L^1_{\text{loc}}(\R^d)$ and the negative part $V_-$ is in the Kato class \cite{Ouh}.
Then $A + wI$ satisfies the upper Gaussian bound for $w > 0$ large enough \cite[(5)]{Ouh}
so that $A + wI$ has a bounded $\Ma$ calculus by Lemma \ref{Lem Criteria Examples Ma calculus} (1), and the norm descriptions for the inhomogeneous spaces hold in this case.

Further Schr\"odinger operators with $\Ma$ calculus such as (small perturbations) of the harmonic oscillator $A = - \Delta + |x|^2,$ a twisted Laplace operator and scattering operators are considered in \cite[Section 7]{DuOS}.

In the recent paper \cite{Tang} there is an extension of the above to weighted $L^p$ spaces.
Suppose that $A$ is a self-adjoint operator such that the corresponding heat kernel satisfies 
\[ |p_t(x,y)| \leq C_N t^{-n/2} \left( 1 + \frac{\sqrt{t}}{\rho(x)} + \frac{\sqrt{t}}{\rho(y)} \right)^{-N} \exp(-b \frac{|x-y|^2}{t}) \]
for all $t > 0,\: N > 0$ and $x,y \in \R^n,$ where the auxiliary function $\rho$ is subject to a further condition \cite[Lemma 2.1]{Tang}.
Then $A$ has a $\Ma$ calculus on $L^p(w)$ where $1 < p < \infty$ and $w$ belongs to some Muckenhoupt weight class $A^{\rho,\infty}_p$ \cite[Theorem 4.1]{Tang}.
Note that the heat kernel condition is satisfied for the above magnetic Schr\"odinger operator as soon as the potential $V$ belongs to a reverse H\"older class $RH_{n/2}$.

\subsection{Operators on weighted $L^p$ spaces}

For several of the above examples on $L^p(U,d\mu)$ spaces the operator $A$ has also a $\Ma$ calculus on the weighted space $L^p(U,w d\mu),$ where $w$ belongs to a Muckenhoupt class.
We refer to \cite[Section 6]{DSY} for details.
This holds in particular for the Laplace operator acting on a homogeneous Lie group,
for a general non-negative self-adjoint elliptic operator on a compact Riemannian manifold, for Laplace operators on irregular domains with Dirichlet boundary conditions and for Schr\"odinger operators with standard or with electromagnetic potential.

Noet that these operators on $L^2$ spaces with $A_2$ weights are usually not self-adjoint.

\subsection{Higher order operators and Schr\"odinger operators with singular potentials on $\R^D$}

In \cite{Bluna} and \cite{Uhl}, operators $A$ that have a $\Ma$ calculus on $L^p(U)$ for $p \in (p_0,p_0')$ are considered.
These include higher order operators with bounded coefficients and Dirichlet boundary conditions on irregular domains.
They are given by a form $a : V \times V \to \C$ of the type
\[ a(u,v) = \int_\Omega \sum_{|\alpha| = |\beta| = k} a_{\alpha\beta} \partial^\alpha u \overline{\partial^\beta v} dx, \]
where $V = \dot{H}^k(\Omega)$ for some arbitrary (irregular) domain $\Omega \subset \R^D.$
One assumes that $a_{\alpha\beta} = \overline{a_{\beta \alpha}} \in L^\infty(\R^D)$ for all $\alpha,\beta$ and Garding's inequality
\[ a(u,u) \geq \delta \| \nabla^k u\|_2^2\text{ for all }u \in V,\]
for some $\delta > 0$ and $\|\nabla^k u\|_2^2 := \sum_{|\alpha| = k} \|\partial^\alpha u\|_2^2.$
Then $a$ is a closed symmetric form and the associated operator $A$ falls in our scope with $p_0 = \max(\frac{2D}{m+D},1)$ \cite{Bluna}.
A further example are Schr\"odinger operators with singular potentials on $\R^D.$
Here $A = -\Delta + V$ on $\R^D$ for $D \geq 3$ where $V = V_+ - V_-,\,V_\pm \geq 0$ are locally integrable and $V_+$ is bounded for simplicity.
We assume the following form bound:
\[ \spr{V_- u}{u} \leq \gamma(\|\nabla u\|_2^2 + \spr{V_+ u}{u}) + c(\gamma) \|u\|_2^2 \text{ for all }u \in H^1(\R^D)\]
and some $\gamma \in (0,1).$
Then the form sum $A = (- \Delta + V_+) - V_-$ is defined and the associated form is closed and symmetric with form domain $H^1(\R^D).$
By standard arguments using ellipticity and Sobolev inequality,
$A$ falls in our scope with $p_0 = \frac{2 D}{D + 2}$ \cite{Bluna}.
We also refer to Auscher's memoir \cite{A} for more on this subject.

\subsection{Laplace operator on graphs}

Consider a countable infinite set $\Gamma$ and let $\sigma(x,y)$ be a weight on $\Gamma$ satisfying $\sigma(x,y) = \sigma(y,x) \geq 0$ and $\sigma(x,x) > 0,\: x,y \in \Gamma.$
This weight induces a graph structure on $\Gamma.$
Assume that $\Gamma$ is connected, where $x$ and $y$ are neighbors if $\sigma(x,y) > 0.$
We consider the discrete measure $\mu$ defined by
$ \mu(\{x\}) = \sum_{y \text{ neighbor } x} \sigma(x,y)$ and the transition kernel $p(x,y) = \frac{\sigma(x,y)}{\mu(x)\mu(y)}.$
Then $A = I - P,$ where $P f(x) = \sum_y p(x,y) f(y) \mu(y)$ is the discrete Laplacian on the graph $\Gamma.$
According to \cite[Subsection 1.3]{KM}, $A$ has a $\Ma$ calculus on $H^p$ for $p_0<p\leq 1$ and on $L^p$ for $1 < p < \infty$ with $p_0 = \frac{D}{D + \beta}$ and $\alpha> D ( \frac1p - \frac12),$ where $D$ and $\beta$ are constants depending on the graph \cite[Section 1]{KM},
so that our results are available in this case.

\subsection{Laplace operators on fractals}

Many interesting examples of spaces and operators that fit into our context are described in the theory of Brownian Motion on fractals (see for example \cite{Kig}).
We mention only the Laplace operator $A$ on the Sierpinski Gasket,
which has a $\Ma$ calculus for $\alpha > \log 3 / (\log 5 - \log 3)$ \cite{BP}.
Thus our results from Sections \ref{Sec 3 Spectral Decomposition} and \ref{Sec 5 Real Interpolation} are available for this operator.

\subsection{Differential operators of Hermite and Laguerre type}

These differential operators are of the form $Af = \sum_n \lambda_n \spr{f}{h_n} h_n,$
where $(h_n)$ is an orthonormal basis in $L^2(U,\mu).$
In particular,
\begin{enumerate}
\item \textit{Hermite} $h_n(x) = (2^n n! \sqrt{\pi})^{-\frac12} (-1)^n \frac{d^n}{dx^n}(e^{-x^2}) e^{x^2/2},$
$A f = \sum_n (2n + d) \spr{f}{h_n} h_n$ on $L^2(\R^d,dx).$
\item \textit{Laguerre} We define first 
\[L^\alpha_k(x) = \sum_{j=0}^k \frac{\Gamma(k+\alpha + 1)}{\Gamma(k-j+1)\Gamma(j+\alpha+1)}\frac{(-x)^j}{j!}\] and \[\mathcal{L}^\alpha_k(x) = \left(\frac{\Gamma(k+1)}{\Gamma(k+\alpha+1)}\right)^{\frac12}e^{-x}x^{\alpha/2}L^\alpha_k(x)\]
and then $h_n^\alpha(x) = \mathcal{L}^\alpha_k(x^2)(2x)^{\frac12}.$
Then the Laguerre operator is defined by $A f= \sum_{k=0}^\infty (2k+1) \spr{f}{h_k^\alpha}h_k^\alpha$ on $L^2(\R_+).$
\end{enumerate}

Since the spectrum of these operators is discrete the Paley-Littlewood decomposition may take the form 
\begin{equation}\label{Equ Hermite}\|A^\theta f\| \cong \| ( \sum_n |2^{n \theta}P_n f|^2 )^{\frac12} \|_{L^p}
\end{equation}
or 
\[\|f\|_{\theta,q} \cong \left( \sum_n 2^{\theta n q} \|P_n f\|_{L^p}^q \right)^{1/q},\]
where $P_n f = \sum_{n \in I_n} \spr{f}{h_n} h_n$ with $I_n = \{ \lambda_k:\: \lambda_k \in [2^n,2^{n+1}] \}$ \cite{Dz}, provided the required multiplier theorems are true.
However these are shown for the Hermite operator in \cite[p.~91]{Than},
where also a weight $w(x) = |x|^{-n(p/2-1)}$ is allowed for the range $\frac43 < p < 4,$ and for the Laguerre operator in \cite[p.~159]{Than}.
For the Hermite operator the Paley-Littlewood decomposition \eqref{Equ Hermite} was already shown by Zheng \cite{Zha} with a different proof.
Moreover, in \cite[Theorem 1.2]{Epp}, \eqref{Equ Hermite} is proved for $\theta = 0$ for slightly more general spectral multipliers $Q(2^{-n}A)$ in place of $\dyad_n(A),$ allowing $\sum_{n=0}^\infty |Q(2^{-n}t)| \cong 1$ for $t > 0$ in place of exactly $=1.$
The $\Mih^1$ calculus on $X_\theta$ and more general Triebel-Lizorkin type spaces associated with the Hermite operator is proved in \cite[Theorem 1]{Epp1}.

\subsection{Bessel operator}

The Bessel operator on $L^p(\R_+,d\mu(x))$ where $d\mu(x) = x^r dx$ for $r > 0$ is defined by $A = - \left( \frac{d^2}{dx^2} + \frac{r}{x} \frac{d}{dx} \right).$
By \cite{GS}, this operator has a $\Ma$ calculus for $\alpha > \frac{r+1}{2},$ so that our results are available in this case.

\subsection{The Dunkl Transform}

The Dunkl transform has been introduced e.g. in \cite{Solt,Ank}.
Take the weight $w(x) = \prod_{j=1}^d |x_j|^{2k + 1}$ for some $k \geq -\frac12$ and consider the Dunkl transform isometry
$\mathcal{F}_k : L^2(\R^d,w) \to L^2(\R^d,w)$ as introduced in \cite[p.4]{Ank}.
Furthermore, for a given $k \geq - \frac12$ and $j = 1,\ldots,d$, we put
\[T_j f (x) = \frac{\partial}{\partial x_j} f(x) + \frac{2 k + 1}{x_j} \left[ \frac{f(x) - f(x_1,x_2,\ldots,x_{j-1},-x_j,x_{j+1},\ldots,x_d)}{2} \right] \]
\cite[p.3]{Ank} and then define the Dunkl Laplacian by
$\Delta_k = \sum_{j = 1}^d T_j^2 .$
One has the relation $\mathcal{F}_k (\lambda - \Delta_k)f)(x) = (\lambda + |x|^2) \mathcal{F}_k (f) (x) $ for $f \in \mathcal{S}(\R^d)$ and $\lambda > 0.$
Finally, there is a certain Dunkl convolution $\ast_k$ with the two properties
$\mathcal{F}_k(f \ast_k g) = \mathcal{F}_k(f) \mathcal{F}_k (g)$ and $\| f \ast_k g\|_{L^p(\R^d,w)} \leq \|f\|_{L^p(\R^d,w)} \|g\|_{L^1(\R^d,w)}$ \cite[p.5]{Ank}.

\begin{cor}
The Dunkl Operator $-\Delta_k$ is $0$-sectorial and has a Mihlin calculus.
Further the following description of the associated real interpolation norm holds for $\theta > 0,\: 1 \leq p,q \leq \infty$:
\[ \|f \|_{\B^\theta_q} \cong \left( \sum_{j \in \N_0} 2^{j\theta q} \| \phi_j \ast_k f \|_{L^p(\R^d,w)}^q \right)^{\frac1q} \]
(standard modification if $q = \infty$),
where $\dyad(y)= \mathcal{F}_k(\phi_j)(\sqrt{y}),\,(y > 0)$ and $(\dyad_j)_{j \in \N_0}$ is an inhomogeneous dyadic partition of unity on $\R_+.$
\end{cor}

\begin{proof}
According to Lemma \ref{Lem Criteria Examples Ma calculus} (2) and \cite[Corollary 1]{Solt} and the fact that for any $m \in \Ma,$ $m(|\cdot|^2)$ and $m$ have equivalent $\Ma$ norms, $T_j^2$ has a $\Ma$ calculus for $\alpha > k + 1.$
The operators $T_j^2$ are self-adjoint on $L^2(\R,w).$ 
Thus, by a multivariate spectral multiplier theorem \cite[Theorem 4.1]{Wro}, the sum $-\Delta_k = -\sum_{j=1}^d T_j^2$ has a $\Mih^{\alpha d}$ calculus.
By the above properties we have that $\mathcal{F}_k(\phi_j \ast_k f) = \mathcal{F}_k(\phi_j)\mathcal{F}_k(f),$ and it is easy to check that this implies
$\phi_j \ast_k f = \dyad_j(-\Delta_k)f.$
Then the corollary follows from the real interpolation in Subsection \ref{Subsec Schrodinger}.
\end{proof}

Note that the norm of the corollary recovers the Besov-Dunkl norm in \cite{Ank} and
shows that the according spaces form a real interpolation scale.
Our theory also makes Triebel-Lizorkin type decompositions available in this situation.

\subsection{Besov spaces associated with operators satisfying Poisson type estimates}

In this subsection, we compare our Besov type spaces from Section \ref{Sec 5 Real Interpolation} with the recent Besov spaces associated with operators satisfying Poisson type estimates from \cite{BDY}.
In that article, the authors assume that $\X$ is a quasi-metric measure space such that the measure $\mu$ is subpolynomial, $\mu(B(x,r)) \leq C r^n$ for some $n > 0$ and any $r > 0,$ where $B(x,r) = \{y \in \X :\: d(y,x) < r \}$ and $d$ is the quasi-metric of $\X.$
A standing assumption is moreover that $-A$ generates a holomorphic semigroup on $L^2(\X)$ and that the integral kernel $p_t(x,y)$ of $e^{-tA}$ has the following Poisson bounds:
There exists $C > 0$ such that
\begin{equation}\label{equ Poisson estimate}
|p_t(x,y)| \leq C \frac{t}{(t + d(x,y))^{n+1}} \quad (t > 0,\: x,y \in \X).
\end{equation}
We emphasize that the authors do not assume the volume doubling property for the measure $\mu$ and no $\HI$ calculus for $A$ on $L^p(\X).$
However, their Besov spaces are based on an analytic decomposition of unity $\phi(t) = t e^{-t}$
and not a decomposition with a less regular function $\phi$ having compact support, or merely a suitable decay of $\phi(t)$ at $t \to 0$ and $t \to \infty.$
We will show that under our assumptions, our scale of Besov spaces and the scale in \cite{BDY} are the same.
The $\Ma_1$ calculus we need for $A$ on $L^p(\X),\: 1 < p <\infty$ (or on its injective part $\overline{R(A)} \subset L^p(\X)$), is known to exist in the following situations, which also fit \eqref{equ Poisson estimate}.
\begin{enumerate}
\item $A$ is a classical, strongly elliptic pseudo-differential operator of order $1$ on a compact Riemannian manifold $M = \X$ equipped with the Riemannian metric, such that $\gamma(A) = \inf\{ \Re \lambda :\: \lambda \in \sigma(A) \} > 0.$
Indeed, in this case, \cite[Theorem 3.14]{GG} gives the Poisson estimates \eqref{equ Poisson estimate}.
Moreover, \cite{SS} shows a H\"ormander functional calculus for $A$ on $L^p(\X),$ which contains the weaker $\Ma_1$ calculus for some $\alpha > 0.$
\item $A$ satisfies Gaussian estimates in the sense of \cite[(7.3)]{Ouha},
is self-adjoint on $L^2(\X)$ and $\mu$ satisfies the volume doubling property.
Indeed, the $\Ma_1$ calculus for $\alpha > d |\frac{1}{2}-\frac{1}{p}| + 1$ follows from the $L^p$ estimate of the complex time semigroup from \cite[Theorem 7.4]{Ouha} and Remark \ref{rem N_T}.
Note that as soon as $\mu$ has the volume doubling property, 2. covers all the examples (ii),(iii),(iv),(v) from \cite[p.~2456]{BDY}, according to \cite[p.~194-195]{Ouha} and also example (i) if the coefficient functions are real valued and smooth enough \cite[p.~195]{Ouha}.
\end{enumerate}

\begin{prop}
Let $(\X,\mu,d)$ and $A$ be as in the beginning of this subsection and let $A$ have a $\Ma_1$ calculus on $L^p(\X)$ for some $1 < p < \infty$ (or on $\overline{R(A)} \subset L^p(\X)$ if $A$ is not injective).
Let $-1 < \theta < 1$ and $1 \leq q < \infty.$
Then the Besov type space $\dot{B}^{\theta,A}_{p,q}$ from \cite{BDY} coincides with our Besov type space $\dB^\theta_q$ (or with $\dB^\theta_q \oplus N(A) \subset \overline{R(A)} \oplus N(A)$ if $A$ is not injective).
\end{prop}

\begin{proof}
Let $(\ddyad_n)_{n \in \Z}$ be a dyadic partition of unity.
According to Theorem \ref{Thm Besov identification continuous} and Remark \ref{Rem Besov continuous}, we have for $f \in \dB^\theta_q,$
\begin{equation}
\label{equ proof BDY}
\int_0^\infty \left( t^{-\theta} \|\ddyad_0(tA) f\|_p \right)^q \frac{dt}{t} \cong \int_0^\infty \left( t^{-\theta} \|tA e^{-tA} f\|_p \right)^q \frac{dt}{t}.
\end{equation}
Note that for $N \in \N$ and $\theta < N,$ a similar formula as \eqref{equ proof BDY} holds where $tA e^{-tA}$ is replaced by $(tA)^N e^{-tA}.$

We next show that $\dB^\theta_q$ is contained in $\dot{B}^{\theta,A}_{p,q}.$
To this end, it suffices by density to show that $D_A \subset \dot{B}^{\theta,A}_{p,q}$ and that $\|f\|_{\dot{B}^{\theta,A}_{p,q}} \lesssim \|f\|_{\dB^\theta_q}$ for any $f \in D_A.$
We show that $f \in (\mathcal{M}^{-\theta,A'}_{p',q'})',$ the latter space being defined in \cite[p.~2465-2466]{BDY}.
Let $h \in \mathcal{M}^{-\theta,A'}_{p',q'}.$
Then $\int_{\X} |h(x)|^{p'} d\mu(x) \lesssim \int_{\X} \frac{1}{(1+d(x,x_0))^{(n+\epsilon)p'}} d\mu(x) < \infty,$ so that $h$ belongs to $L^{p'}(\X).$
Note that if $A$ has a $\Ma_1$ calculus on $\overline{R(A)} \subset L^p(\X),$ then $A'$ has a $\Ma_1$ calculus on $\overline{R(A')} \subset L^{p'}(\X).$
Indeed, we have $R(\lambda,A)' = R(\lambda,A'),$ so that $\phi(A)' = \phi(A')$ for any $\phi \in \HI_0(\Sigma_\omega),\:\omega \in (0,\pi).$
Then by density of $\HI_0(\Sigma_\omega)$ in $\Ma_1$ by Lemma \ref{Lem HI0 dense in l1Ma}, we deduce that $A'$ has a $\Ma_1$ calculus and $\phi(A') \oplus 0 = (\phi(A) \oplus 0)'$ for any $\phi \in \Ma_1.$

It now follows that since $h \in L^{p'}(\X),$
\begin{align*}
|\langle f, h \rangle| & = \left| \sum_{n \in F} \langle \ddyad_n(A) f, h \rangle \right| \leq \sum_{n \in F} |\langle \ddyad_n(A) \widetilde{\ddyad_n(A)} f, h \rangle| \\
& = \sum_{n \in F} |\langle \widetilde{\ddyad_n}(A) f, \ddyad_n(A') h \rangle| \leq \sum_{n \in F} \|\widetilde{\ddyad_n}(A)f\|_p \|\ddyad_n(A')h\|_{p'},
\end{align*}
where the sum over $n \in F$ is finite, since $f$ is assumed to be in $D_A.$
Now $\ddyad_n(A')h = \widetilde{\ddyad_0}(tA')\ddyad_n(A')h$ for $t$ belonging to a small interval $I$ around $2^{-n}.$
Thus, $\ddyad_n(A')h = c_I \int_I \widetilde{\ddyad_0}(tA') \ddyad_n(A') h \frac{dt}{t},$
and therefore,
\begin{align*}
\| \ddyad_n(A')h\|_{p'} & \lesssim_n \int_I \|\ddyad_n(A')\| t^{\theta} \|\widetilde{\ddyad_0}(tA')h\|_{p'} \frac{dt}{t} \\
& \lesssim \left\{ \int_I (t^\theta \|\widetilde{\ddyad_0}(tA')h\|)^{q'} \frac{dt}{t} \right\}^{\frac{1}{q'}} \\
& \lesssim \|h\|_{\dot{B}^{-\theta,A'}_{p',q'}}.
\end{align*}
This shows $f \in (\mathcal{M}^{-\theta,A'}_{p',q'})',$ and hence, $D_A \subset \mathcal{M}^{-\theta,A'}_{p',q'}.$
We have by \eqref{equ proof BDY} that the resulting embedding $D_A \hookrightarrow \dot{B}^{\theta,A}_{p,q}$ is continuous.
It follows that $\dB^{\theta}_q \hookrightarrow \dot{B}^{\theta,A}_{p,q}.$

We now show the inclusion $\dot{B}^{\theta,A}_{p,q} \subset \dB^\theta_q.$
To this end, we show first that $\dot{B}^{\theta,A}_{p,q} \cap L^p(\X)$ is dense in $\dot{B}^{\theta,A}_{p,q}$ and second that $\dot{B}^{\theta,A}_{p,q} \cap L^p(\X) \subset \dB^\theta_q$ is a continuous injection.
Then the claimed inclusion $\dot{B}^{\theta,A}_{p,q} \subset \dB^\theta_q$ will follow immediately.
So first, let $f \in \dot{B}^{\theta,A}_{p,q}.$
For $N \in \N,$ write in short $\psi_N = \sum_{n = -N}^N \ddyad_n.$
We show that $\psi_N(A)f,$ defined via $(L^p(\X),L^{p'}(\X))$-duality $\langle \psi_N(A)f, h\rangle = \langle f, \psi_N(A')h \rangle,$ is a well-defined element of $\dot{B}^{\theta,A}_{p,q} \cap L^p(\X).$
We have for arbitrary $t > 0,$ $|\langle f, \psi_N(A') h \rangle| = |\langle tAe^{-tA} f, \frac{1}{t \lambda e^{-t\lambda}} \psi_N(\lambda)|_{\lambda = A'} h \rangle|$ and by \cite[(2.4)]{BDY}, $\|tA e^{-tA}f\|_p \lesssim \|sAe^{-sA}f\|_p$ for $\frac12 t \leq s \leq t,$ so that
$\|tA e^{-tA}f\|_p \lesssim \int_{\frac12 t}^t \|sA e^{-sA} f\|_p \frac{ds}{s} \lesssim_t \|f\|_{\dot{B}^{\theta,A}_{p,q}}.$
Thus, $|\langle f, \psi_N(A')h \rangle| \lesssim \|f\|_{\dot{B}^{\theta,A}_{p,q}} \|\frac{1}{t\lambda e^{-t\lambda}} \psi_N(\lambda)\|_{\Ma_1} \|h\|_{p'},$ and therefore, $\psi_N(A)f$ belongs to $L^p(\X).$
Now by \cite[Theorem 3.4 and (3.5)]{BDY}, $|\langle f, \psi_N(A')h \rangle| \lesssim \|f\|_{\dot{B}^{\theta,A}_{p,q}} \|\psi_N(A') h\|_{\dot{B}^{-\theta,A'}_{p',q'}} \lesssim \|f\|_{\dot{B}^{\theta,A}_{p,q}} \|h\|_{\mathcal{M}^{-\theta,A'}_{p',q'}}.$
Hence, $\psi_N(A)f \in (\mathcal{M}^{-\theta,A'}_{p',q'})'.$
Moreover,
\begin{align*}
\|tA e^{-tA} \psi_N(A) f\|_p & = \sup_{\|h\|_{p'} \leq 1} | \langle tA e^{-tA} \psi_N(A) f, h \rangle| = \sup | \langle f , \psi_N(A') tA' e^{-tA'} h \rangle| \\
& = \sup | \langle f, tA' e^{-tA'} \psi_N(A') h \rangle| = \sup | \langle tA e^{-tA} f, \psi_N(A') h \rangle| \\
& \leq \|tA e^{-tA} f\|_p \|\psi_N(A')\|_{p' \to p'} \| = \|\psi_N(A)\|_{p \to p} \|tA e^{-tA} f\|_p.
\end{align*}
It follows easily that $\|\psi_N(A) f\|_{\dot{B}^{\theta,A}_{p,q}} \leq \|\psi_N(A)\|_{p \to p} \|f\|_{\dot{B}^{\theta,A}_{p,q}}$ and hence, $\psi_N(A) f$ belongs to $\dot{B}^{\theta,A}_{p,q}.$
We next show that $\psi_N(A)f$ approximates $f$ in $\dot{B}^{\theta,A}_{p,q}.$
According to \cite[Proposition 4.4]{BDY}, we have
\begin{align*}
\|\psi_N(A)f - f \|_{\dot{B}^{\theta,A}_{p,q}}^q & \cong \int_0^\infty \left( t^{-\theta} \| t^2 A^2 e^{-tA} (\psi_N(A) - 1) f\|_p \right)^q \frac{dt}{t} \\
& = \int_0^\epsilon \ldots \frac{dt}{t} + \int_\epsilon^{\frac{1}{\epsilon}} \ldots \frac{dt}{t} + \int_{\frac{1}{\epsilon}}^\infty \ldots \frac{dt}{t},
\end{align*}
for fixed $\epsilon \in (0,1).$
The integrand is dominated by $\left(t^{-\theta} \|tA e^{-tA/2}(\psi_N(A)-1)\|_{p\to p} \|tA e^{-tA/2} f\| \right)^q.$
Note that by an elementary calculation similar to one performed previously in this proof, $\sup_{t>0,N\in\N}\|t\lambda e^{-t\lambda/2}(\psi_N(\lambda) -1)\|_{\Ma_1} <\infty,$ so that the first and the third integral above become small when $\epsilon$ is sufficiently close to $0.$
Furthermore, if $\epsilon$ is fixed, then $\|t\lambda e^{-t\lambda/2} (\psi_N(\lambda)-1)\|_{\Ma_1} \to 0$ as $N \to \infty,$ uniformly in $t \in [\epsilon,\frac{1}{\epsilon}].$
Thus, also the second integral becomes small for $N$ sufficiently large.
Resuming the above, we have shown that $\dot{B}^{\theta,A}_{p,q} \cap L^p(\X)$ is dense in $\dot{B}^{\theta,A}_{p,q}.$
At last, the inclusion $L^p(\X) \subset \dot{X}_1 + \dot{X}_{-1}$ together with \eqref{equ proof BDY} for functions in $L^p(\X)$ gives the desired injection $L^p(\X) \cap \dot{B}^{\theta,A}_{p,q} \hookrightarrow \dB^{\theta}_q.$
\end{proof}

\subsection{A non self-adjoint example: Lam\'e system of elasticity}

In this subsection, we cite an example of an operator $A$ which is $0$-sectorial on $X =L^p(\R^d;\C^{d+1})$ for any $1 < p < \infty,$ has an $\Ma$ calculus on $X,$ but is not self-adjoint on $L^2(\R^d;\C^{d+1}).$
Of course, the results on Paley-Littlewood decomposition and complex interpolation spaces from Section \ref{Sec 3 Spectral Decomposition}, and real interpolation spaces from Section \ref{Sec 5 Real Interpolation} apply in full strength for $A$ on $X.$

Consider the so-called Lam\'e operator on $\R^{d+1}$ with $d \in \N,$ which has the form
\[ Lu = \mu \Delta u + (\lambda + \mu) \nabla \,\text{div}\,u, \quad u = (u_1,\ldots,u_{d+1}), \]
where the constants $\lambda, \mu \in \R$ (typically called Lam\'e moduli) are assumed to satisfy $\mu > 0$ and $2 \mu + \lambda > 0.$
The associated Dirichlet problem
\[
\begin{cases} 
Lu & = 0 \text{ in }\R^d\times [0,\infty)\\
u|_{\R^d \times \{ 0 \}} &= f \in L^p(\R^d;\C^{d+1})
\end{cases}
\]
has the solution $u(x,t) = T_tf(x)$ given by
\begin{align*}
& (T_tf)_\beta(x) = \frac{4\mu}{3\mu + \lambda}\frac{1}{\omega_d} \int_{\R^d} \frac{t}{(|x-y|^2 + t^2)^{\frac{d+1}{2}}}f_\beta(y) dy \\
& + \frac{\mu + \lambda}{3 \mu + \lambda} \frac{2(d+1)}{\omega_d} \sum_{\gamma=1}^{d+1}\int_{\R^d} \frac{t(x-y,t)_\beta (x-y,t)_\gamma}{(|x-y|^2 + t^2)^{\frac{d+3}{2}}}f_\gamma(y) dy \quad (\beta = 1,\ldots,d+1),
\end{align*}
where $\omega_d$ is the area of the unit sphere $S^d$ in $\R^{d+1},$ and $(x-y,t)_\beta = x_\beta - y_\beta$ if $\beta \in \{ 1 , \ldots , d \}$ and $(x-y,t)_{d+1} = t$ \cite[Theorem 5.2]{MMMM}.
$T_t$ is a semigroup on $L^p(\R^d;\C^{d+1})$ for $1 < p < \infty,$ since the operator $L$ has constant coefficients.
It is strongly continuous and according to \cite[Corollary 4.2]{KrPoiss}, its negative generator $A$ has a H\"ormander calculus, which implies the $\Ma$ calculus \cite[Proposition 4.9]{Kr}, for $\alpha > \frac{d}{2} + 1.$

Further, for $p = 2,$ the semigroup operators $T_t$ are not self-adjoint, and thus, $A$ cannot be self-adjoint on $L^2(\R^d;\C^{d+1}).$
Indeed, if we write $(T_tf)_\beta (x) = \sum_{\gamma = 1}^{d+1} \int_{\R^d} k_t^{\beta,\gamma}(x-y)f_\gamma(y) dy,$ 
then for a function $f$ with $f_\gamma = 0$ for $\gamma \geq 2,$ we have for $\beta = d+1:$
$(T_tf)_{d+1}(x) = \int_{\R^d} k_t^{d+1,1}(x-y)f_1(y) dy$ and $(T_t^*f)_{d+1}(x) = \int_{\R^d} k_t^{1,d+1}(y-x)f_1(y) dy.$
However, $k_t^{d+1,1}(x-y)$ is the negative of $k_t^{1,d+1}(y-x),$ as one checks easily,
so $T_tf$ and $T_t^*f$ do not coincide in general.

\section{Bisectorial Operators}\label{Sec 6 Bisectorial Operators}

In this short section we indicate how to extend our results to bisectorial operators.
An operator $A$ with dense domain on a Banach space $X$ is called bisectorial of angle $\omega \in [0,\frac{\pi}{2})$ if it is closed, its spectrum is contained in the closure of $S_\omega = \{ z \in \C :\: |\arg(\pm z)| < \omega \},$ and one has the resolvent estimate 
\[\|(I+\lambda A)^{-1}\|_{B(X)} \leq C_{\omega'},\: \forall \: \lambda \not\in S_{\omega'},\: \omega' > \omega .\]
If $X$ is reflexive, then for such an operator we have again a decomposition $X = N(A) \oplus \overline{R(A)},$ so that we may assume that $A$ is injective.
The $\HI(S_\omega)$ calculus is defined as in \eqref{Equ Cauchy Integral Formula}, but now we integrate over the boundary of the double sector $S_\omega.$
If $A$ has a bounded $\HI(S_\omega)$-calculus, or more generally, if we have $\|Ax\| \cong \|(-A^2)^{\frac12}x\|$ for $x \in D(A) = D((-A^2)^{\frac12})$ (see e.g. \cite{DW}), then the spectral projections $P_1,\: P_2$ with respect to $\Sigma_1 = S_\omega \cap \C_+,\: \Sigma_2 = S_\omega \cap \C_-$ give a decomposition $X = X_1 \oplus X_2$ of $X$ into invariant subspaces for resolvents of $A$ such that the restriction $A_1$ of $A$ to $X_1$ and $-A_2$ of $-A$ to $X_2$ are sectorial operators with $\sigma(A_i) \subset \Sigma_i.$
For $f \in \HI_0(S_\omega)$ we have 
\begin{equation}\label{Equ Bisectorial}
f(A)x = f|_{\Sigma_1}(A_1)P_1 x + f|_{\Sigma_2}(A_2)P_2 x.
\end{equation}
We define the Mihlin class $\Ma(\R)$ on $\R$ by $f \in \Ma(\R)$ if $f \chi_{\R_+} \in \Ma$ and $f(-\cdot) \chi_{\R_+} \in \Ma.$
Let $A$ be a $0$-bisectorial operator, i.e. $A$ is $\omega$-bisectorial for all $\omega > 0.$
Then $A$ has a $\Ma(\R)$ calculus if there is a constant $C$ so that $\|f(A)\| \leq C \|f\|_{\Ma(\R)}$ for $f \in \bigcap_{0 < \omega < \pi} \HI(S_\omega) \cap \Ma(\R).$
Clearly, $A$ has a $\Ma(\R)$ calculus if and only if $A_1$ and $-A_2$ have a $\Ma$ calculus and in this case \eqref{Equ Bisectorial} holds again.
Let $\ddyad_n$ or $\dyad_n$ be the dyadic partitions of unity from Definition \ref{Def Partitions of unity} and extend them to $\R$ by $\ddyad_n(t) = \ddyad_n(|t|)$ and $\dyad_n(t) = \dyad_n(|t|)$ for $t \in \R.$

Using the projections $P_1$ and $P_2$ one verifies that
\[ \E \| \sum_n \epsilon_n \dyad_n(A) x \| \cong \E \| \sum_n \epsilon_n \dyad_n(A_1)P_1 x \| + \E \| \sum_n \epsilon_n \dyad_n(A_2)P_2 x \|. \]
If $\psi : (0,\infty) \to \C$ is as in Theorem \ref{Thm Fractional power continuous}, put $\psi(t) = \psi(|t|)$ for $t \in \R$ and obtain
\[ \| t^{-\theta} \psi(tA) x \|_{\gamma(\R,\frac{dt}{t},X)} \cong \| t^{-\theta} \psi(tA_1)P_1 x \|_{\gamma(\R_+,\frac{dt}{t},X_1)} + \| t^{-\theta} \psi(tA_2)P_2 x \|_{\gamma(\R_+,\frac{dt}{t},X_2)} .\]
Similar statements are true for $\ddyad_n$ and the Besov norms.
Now it is clear how our results from Sections \ref{Sec 3 Spectral Decomposition} and \ref{Sec 5 Real Interpolation} extend to bisectorial operators.

\section{Strip-type Operators}\label{Sec 6 Strip-type Operators}

The spectral decomposition from Theorem \ref{Thm PL Decomposition} and Proposition \ref{Prop PL extended} can be stated more naturally in the context of $0$-strip-type operators.
For $\omega > 0$ we let $\Str_\omega = \{z \in \C :\: |\Im z| < \omega\}$ the horizontal strip of height $2 \omega.$
We further define $\HI(\Str_\omega)$ to be the space of bounded holomorphic functions on $\Str_\omega,$ which is a Banach algebra equipped with the norm $\|f\|_{\infty,\omega} = \sup_{\lambda \in \Str_\omega} |f(\lambda)|.$
A densely defined operator $B$ is called $\omega$-strip-type operator if $\sigma(B) \subset \overline{\Str_\omega}$ and for all $\theta > \omega$ there is a $C_\theta > 0$ such that $\|\lambda (\lambda - B)^{-1}\| \leq C_\theta$ for all $\lambda \in \overline{\Str_\theta}^c.$
Similarly to the sectorial case, one defines $f(B)$ for $f \in \HI(\Str_\theta)$ satisfying a decay at $|\Re \lambda| \to \infty$ by a Cauchy integral formula, and says that $B$ has a bounded $\HI(\Str_\theta)$ calculus provided that $\|f(B)\| \leq C \|f\|_{\infty,\theta},$ in which case $f \mapsto f(B)$ extends to a bounded homomorphism $\HI(\Str_\theta) \to B(X).$
We refer to \cite{CDMY} and \cite[Chapter 4]{Haasa} for details.
We call $B$ $0$-strip-type if $B$ is $\omega$-strip-type for all $\omega > 0.$

There is an analogous statement to Lemma \ref{Lem Hol} which holds for a $0$-strip-type operator $B$ and $\Str_\omega$ in place of $A$ and $\Sigma_\omega,$ and $\Hol(\Str_\omega) = \{ f : \Str_\omega \to \C :\: \exists n \in \N :\: (\rho \circ \exp)^n f \in \HI(\Str_\omega) \},$
where $\rho(\lambda) = \lambda ( 1 + \lambda)^{-2}.$

In fact, $0$-strip-type operators and $0$-sectorial operators with bounded $\HI(\Str_\omega)$ and bounded $\HI(\Sigma_\omega)$ calculus are in one-one correspondence by the following lemma. 
For a proof we refer to \cite[Proposition 5.3.3., Theorem 4.3.1 and Theorem 4.2.4, Lemma 3.5.1]{Haasa}.

\begin{lem}
Let $B$ be a $0$-strip-type operator and assume that there exists a $0$-sectorial operator $A$ such that $B = \log(A)$.
This is the case if $B$ has a bounded $\HI(\Str_\omega)$ calculus for some $\omega < \pi.$
Then for any $f \in \bigcup_{0 < \omega < \pi} \Hol(\Str_\omega)$ one has
\[ f(B) = (f\circ \log)(A). \]
Note that the logarithm belongs to $\Hol(\Sigma_\omega)$ for any $\omega \in (0,\pi).$
Conversely, if $A$ is a $0$-sectorial operator that has a bounded $\HI(\Sigma_\omega)$ calculus for some $\omega \in (0,\pi),$ then $B = \log(A)$ is a $0$-strip-type operator.
\end{lem}

Let $B$ be a $0$-strip-type operator and $\alpha > 0.$
We say that $B$ has a (bounded) $\Ba$ calculus if there exists a constant $C > 0$ such that
\[ \|f(B)\| \leq C \|f\|_{\Ba} \quad (f\in \bigcap_{\omega > 0} \HI(\Str_\omega) \cap \Ba) .\]
In this case, by density of $\bigcap_{\omega > 0} \HI(\Str_\omega) \cap \Ba$ in $\Ba$,
the definition of $f(B)$ can be continuously extended to $f \in \Ba.$

Let $\equi \in C^\infty_c(\R).$
Assume that $\supp \equi \subset [-1,1]$ and $\sum_{n = -\infty}^\infty \equi(t-n) = 1$ for all $t \in \R.$
For $n \in \Z,$ we put $\equi_n = \equi(\cdot - n)$ and call $(\equi_n)_{n \in \Z}$ an equidistant partition of unity.

Assume that $B$ has a $\Ba$ calculus for some $\alpha > 0.$
Let $f \in \Baloc.$
We define the operator $f(B)$ to be the closure of
\[ \begin{cases}
D_B \subset X & \longrightarrow X \\
x & \longmapsto \sum_{n \in \Z} (\equi_n f)(B)x,
\end{cases}
\]
where $D_B = \{ x \in X :\: \exists N \in \N:\: \equi_n(B)x= 0 \quad (|n| \geq N) \}.$

Then there holds a modified version of Proposition \ref{Prop Soaloc calculus}, a proof of which can be found in \cite[Proposition 4.25]{Kr}.
Now the strip-type version of Theorem \ref{Thm PL Decomposition} reads as follows.

\begin{thm}
Let $B$ be a $0$-strip-type operator having a $\Ba$ calculus for some $\alpha > 0.$
Let further $(\equi_n)_{n \in \Z}$ be an equidistant partition of unity and put $\tequi_n = \equi_n$ for $n \geq 1$ and $\tequi_n = \sum_{k = - \infty}^0 \equi_k$ for $n = 0.$
The norm on $X$ has the equivalent descriptions:
\begin{align*}
\|x\| \cong \E \bignorm{ \sum_{n \in \Z} \epsilon_n \equi_n(B) x} \cong \sup\left\{ \bignorm{\sum_{n \in \Z} a_n \equi_n(B) x} :\: |a_n| \leq 1 \right\} \\
\intertext{and}
\|x\| \cong \E \bignorm{ \sum_{n \in \N_0} \epsilon_n \tequi_n(B) x} \cong \sup\left\{ \bignorm{\sum_{n \in \N_0} a_n \tequi_n(B) x} :\: |a_n| \leq 1 \right\}.
\end{align*}
\end{thm}

The strip-type version of Proposition \ref{Prop PL extended} is the following.

\begin{prop}
Let $B$ be a $0$-strip-type operator having a $\Ba$ calculus.
Further let $g \in \Baloc$ such that $g$ is invertible and $g^{-1}$ also belongs to $\Baloc.$
Assume that for some $\beta > \alpha,$
\[ \sup_{n \in \Z} \| \tequi_n g \|_{\Bes^\beta_{\infty,\infty}} \cdot \| \equi_n g^{-1} \|_{\Bes^\beta_{\infty,\infty}} < \infty. \]
Let $(c_n)_{n \in \Z}$ be a sequence in $\C \backslash \{ 0 \}$ satisfying $|c_n| \cong \| \tequi_n g \|_{\Bes^\beta_{\infty,\infty}}.$
Then for any $x \in D(g(B)),$ $\sum_{n \in \Z} c_n \equi_n(B)x$ converges unconditionally in $X$ and
\[ \|g(B) x\| \cong \E \bignorm{ \epsilon_n c_n \equi_n(B) x} \cong \sup\left\{ \bignorm{ \sum_{n \in \Z} a_n c_n \equi_n(B) x} : \: |a_n| \leq 1 \right\}. \]
\end{prop}

For a representation of the Besov type space norm for the operator $B$ we refer to \cite[Section 3.6]{AmBG}.

\end{document}